\documentclass[12pt]{amsart}
\usepackage{amscd, amssymb,latexsym,amsmath, amscd, amsmath}
\usepackage[all]{xy}
\usepackage{anysize}
\usepackage[sc]{mathpazo} 
\usepackage{color}
\marginsize{2.5cm}{2.5cm}{2.5cm}{2.5cm}
\newtheorem{theorem}{Theorem}[section] 
\newtheorem{lemma}[theorem]{Lemma}
\newtheorem{corollary}[theorem]{Corollary}
\newtheorem{proposition}[theorem]{Proposition}
\theoremstyle{definition}

\theoremstyle{remark}
\newtheorem{remark}{Remark}

\newcommand{\GL}{\rm GL}

\newcommand{\Ker}{{\rm Ker}}

\begin{document}

\title[Iwahori-Hecke model for mod $p$ representations of ${\rm GL}_2(F)$]{Iwahori-Hecke model for mod $p$ representations of ${\rm GL}_2(F)$}
\author{U. K. Anandavardhanan and Arindam Jana}

\address{Department of Mathematics, Indian Institute of Technology Bombay, Mumbai - 400076, India.}
\email{anand@math.iitb.ac.in}

\address{Department of Mathematics, Indian Institute of Technology Bombay, Mumbai - 400076, India.}
\email{arindam@math.iitb.ac.in}


\date{}

\begin{abstract}
For a $p$-adic field $F$, the space of pro-$p$-Iwahori invariants of a universal supersingular mod $p$ representation $\tau$ of ${\rm GL}_2(F)$ is determined in the works of Breuil, Schein, and Hendel. The representation $\tau$ is introduced by Barthel and Livn\'e and this is defined in terms of the spherical Hecke operator. In \cite{ab13,ab15}, an Iwahori-Hecke approach was introduced to study these universal supersingular representations in which they can be characterized via the Iwahori-Hecke operators. In this paper, we construct a certain quotient $\pi$ of $\tau$, making use of the Iwahori-Hecke operators. When $F$ is not totally ramified over $\mathbb Q_p$, the representation $\pi$ is a non-trivial quotient of $\tau$. We determine a basis for the space of invariants of $\pi$ under the pro-p Iwahori subgroup. A pleasant feature of this "new" representation $\pi$ is that its space of pro-$p$-Iwahori invariants admits a more uniform description vis-\`a-vis the description of the space of pro-$p$-Iwahori invariants of $\tau$.         
\end{abstract}

\maketitle

\section{Introduction}\label{introduction}

For a $p$-adic field $F$, the study of irreducible smooth mod $p$ representations of $\GL_2(F)$ started with the famous work of Barthel and Livn\'e \cite{bl94}. They showed that there exist irreducible smooth representations, called supersingular representations, which cannot be obtained as a subquotient of a parabolically induced representation. 

It is shown in \cite{bl94} that a supersingular representation can be realized as the quotient of a universal module constructed as follows. Let $G=\GL_2(F)$ and let $K$ be its standard maximal compact subgroup. Let $Z$ denote the center of $G$. For an irreducible representation $\sigma$ of $KZ$, let ind$_{KZ}^G \sigma$ be the representation of $G$ compactly induced from $\sigma$. Its endomorphism algebra is a polynomial algebra in one variable:
\[{\rm End}_G\left({\rm ind}_{KZ}^G \sigma \right) \simeq \overline{\mathbb F}_p[T],\]  
where $T$ is the standard spherical Hecke operator and $\overline{\mathbb F}_p$ denotes an algebraic closure of the finite field $\mathbb F_p$ of $p$ elements \cite[Proposition 8]{bl94}. The universal module in consideration is 
\[\tau = \frac{{\rm ind}_{KZ}^G \sigma}{(T)}\]
and a supersingular representation of $G$ is an irreducible quotient of the universal module for some $\sigma$ of $KZ$ up to a twist by a character \cite{bl94}.  

Explicitly constructing a supersingular representation of $\GL_2(F)$ is a challenging problem when $F \neq \mathbb Q_p$ \cite{bp12}. When $F=\mathbb Q_p$, Breuil proved that the universal representation $\tau$ itself is irreducible \cite[Theorem 1.1]{bre03}. The key step in Breuil's proof of the irreducibility of $\tau$ is the explicit computation of its $I(1)$-invariant space, which is of dimension $2$, where $I(1)$ is the pro-$p$-Iwahori subgroup of $K$ \cite[Theorem 3.2.4]{bre03}. The space of $I(1)$-invariants of $\tau$ is infinite dimensional when $F\neq \mathbb Q_p$. An explicit basis for this infinite dimensional space is computed by Schein when $F$ is totally ramified over $\mathbb Q_p$ \cite[\S 2]{sch11} and by Hendel more generally for any $p$-adic field $F$ \cite[Theorem 1.2]{hen18}.

One can also construct a universal module from the perspective of the Iwahori-Hecke operators instead of the spherical Hecke operator $T$ \cite{ab13,ab15}. For this, instead of doing compact induction from an irreducible representation of $KZ$, we start with a regular character $\chi$ of $IZ$, where $I$ is the Iwahori subgroup $K$, and consider the compactly induced representation ind$_{IZ}^G \chi$. Its endomorphism algebra is \cite[Proposition 13]{bl94}: 
\[{\rm End}_G({\rm ind}_{IZ}^G \chi) 
\simeq\frac{\overline{\mathbb{F}}_p[T_{-1,0,},T_{1,2}]}{(T_{-1,0}T_{1,2},T_{1,2}T_{-1,0})},\]
where $T_{-1,0}$ and $T_{1,2}$ are the Iwahori-Hecke operators. When $F$ is a totally ramified extension of $\mathbb Q_p$, it is proved in \cite[Proposition 3.1 \& Remark 1]{ab15} that the image of one of these operators is equal to the kernel of the other; i.e.,
\begin{equation}\label{imker}
{\rm Im}~ T_{-1,0} = {\rm Ker}~ T_{1,2} ~\&~ {\rm Im}~ T_{1,2} = {\rm Ker}~ T_{-1,0}.
\end{equation}

Let $\mathbb F_q$ be the residue field of $F$ where $q=p^f$. Assume $0<r<q-1$ and write $r=r_0+r_1p+\dots+r_{f-1}p^{f-1}$ with $0 \leq r_i \leq p-1$ for $0 \leq i \leq f-1$. Let 
\[\sigma_r = {\rm Sym}^{r_0} \overline{\mathbb F}_p^2 \otimes {\rm Sym}^{r_1} \overline{\mathbb F}_p^2 \circ {\rm Frob} \otimes \dots \otimes  {\rm Sym}^{r_{f-1}} \overline{\mathbb F}_p^2 \circ {\rm Frob}^{f-1}\] 
be an irreducible representation of $\GL_2(\mathbb F_q)$, where Frob is the Frobenius morphism. We continue to denote the corresponding irreducible representation of $K$, obtained via inflation, by $\sigma_r$. Similarly, let $\chi_r$ be the character of $I$, valued in $\overline{\mathbb F}_p^\times$, obtained via the character of the Borel subgroup of $\GL_2(\mathbb F_q)$ defined by
\[\left(\begin{array}{cc}
a & b\\\
0 & d
\end{array}\right)
\mapsto d^r.\] 

We fix a uniformizing element $\varpi$ of the ring of integers $\mathcal{O}$ of $F$. The representation $\sigma_r$ is treated as a representation of $KZ$ by making diag$(\varpi,\varpi)$ acting trivially and similarly the character $\chi_r$ is treated as a character of $IZ$. 

For $g\in G$ and $v \in \sigma_r$, let $g \otimes v$ be the function in ${\rm ind}_{KZ}^G \sigma_r$ supported on $KZg^{-1}$ that sends $g^{-1}$ to $\sigma_r(k)v$. Similarly, for $g \in G$, by $[g,1]$ we define the function in ${\rm ind}_{IZ}^G \chi_r$ which is supported on $IZg^{-1}$ and sending $g^{-1}$ to $1$. It can be seen that every element of ${\rm ind}_{IZ}^G \chi_r$ (resp. ${\rm ind}_{KZ}^G \sigma_r$) is a finite sum of these type of functions $[g,1]$ (resp. $g \otimes v$).

Now \cite[Theorem 1.1]{ab15} takes the form:
\begin{theorem}\label{iso}
Let $F$ be a finite extension of $\mathbb Q_p$ with residue field $\mathbb F_q$ and residue degree $f$. Let $0 < r < q - 1$ and $r =r_0+r_1p+\dots+r_{f-1}p^{f-1}$ with $0 \leq r_i \leq p - 1$. Then
\[\tau_r = \frac{{\rm ind}_{KZ}^G \sigma_r}{(T)} \simeq \frac{{\rm ind}_{IZ}^G \chi_r}{({\rm Im}~ T_{1,2}, {\rm Ker}~ T_{1,2})}.\]
Moreover, this isomorphism is determined by
\[ {\rm Id}\otimes\bigotimes_{j=0}^{f-1}x_j^{r_j} \mod T \mapsto [\beta,1] \mod ({\rm Im}~ T_{1,2}, {\rm Ker}~ T_{1,2}).\]
\end{theorem}

\begin{remark}
Theorem \ref{iso} is stated and proved in \cite[Theorem 4.1]{ab15} when $F$ is a totally ramified extension of $\mathbb Q_p$ (see \cite[Remark 3]{ab15}) and exactly the same proof goes through in the general case as well. 
\end{remark}

\begin{remark}\label{rmk-hendel}
As mentioned earlier, the space of $I(1)$-invariants of $\tau_r$ is computed by Hendel \cite[Theorem 1.2]{hen18}. Stating an explicit basis for this space involves four cases; (i) $e=1, f=1$, (ii) $e>1, f=1$, (iii) $e=1,f>1$, and (iv) $e>1,f>1$.
\end{remark}

In this paper, we study a new universal representation given by
\[\pi_r = \frac{{\rm ind}_{IZ}^G \chi_r}{({\rm Ker}~ T_{-1,0}, {\rm Ker}~ T_{1,2})}\]
which is a further quotient of $\tau_r$. Note that this representation equals $\tau_r$ when $F$ is totally ramified over $\mathbb Q_p$ by (\ref{imker}). We show that when $F$ is not totally ramified over $\mathbb Q_p$, we have strict containments 
\begin{equation}\label{imker2}
{\rm Im}~ T_{-1,0} \subsetneq {\rm Ker}~ T_{1,2} ~\&~ {\rm Im}~ T_{1,2} \subsetneq {\rm Ker}~ T_{-1,0}.
\end{equation}
and thus we have a new representation to investigate for its properties (cf. Remark \ref{rmk-imker2}). At this stage, we also note that the representation $\pi_r$ is indeed non-trivial (cf. Lemma \ref{kernels sum}). 

The main result of this paper gives an explicit basis for the space of $I(1)$-invariants of $\pi_r$. This space turns out to be infinite dimensional as well as in the case of \cite[Theorem 1.2]{hen18}. However, in this case the basis can be written in a uniform manner whenever $F\neq \mathbb Q_p$. Thus, the statement involves only two cases; (i) $F=\mathbb Q_p$ and (ii) $F \neq \mathbb Q_p$. It is interesting to compare our result with that of Hendel in this aspect (cf. Remark \ref{rmk-hendel}).

In order to state the theorem, we introduce a few more notations. Set $I_0=\{0\}$, and for $n \in \mathbb N$, let 
\[I_n=\left\{[{\mu}_0]+[{\mu}_1]\varpi+\dots+[\mu_{n-1}]{\varpi}^{n-1}\mid
{\mu}_i\in \mathbb{F}_q\right\} \subset \mathcal{O},\]
where, for $x \in \mathbb F_q$, we denote its multiplicative representative in $\mathcal O$ by $[x]$.
If $0\leq m\leq n,$ let $[\cdot]_m:I_n\rightarrow I_m$ be the truncation map defined by
\[\sum\limits_{\substack{i=0}}^{n-1}[\lambda_i]\varpi^i\mapsto \sum\limits_{\substack{i=0}}^{m-1}[\lambda_i]\varpi^i.\]  
Let us denote
\[\alpha=\left(\begin{array}{cc}
1 & 0 \\
0 & \varpi
\end{array}\right),~\beta=\left(\begin{array}{cc}
0      & 1 \\
\varpi & 0
\end{array}\right),~ w=\left(\begin{array}{cc}
0 & 1 \\
1 & 0
\end{array}\right),\]
and observe that $\beta=\alpha w$ normalizes $I(1)$. For any $n \in \mathbb N$, we denote
\begin{align*}
s_n^k &=\sum\limits_{\substack{\mu\in I_n}}\mu_{n-1}^k\left[\left(\begin{array}{cc}
\varpi^n & \mu \\
0    &  1
\end{array}\right),1\right],\\
t_n^s &=\sum\limits_{\substack{\mu\in I_n}}\mu_{n-1}^s\left[\left(\begin{array}{cc}
\varpi^{n-1} & [\mu]_{n-1} \\
0    &  1
\end{array}\right) 
\left(\begin{array}{cc}
1& [\mu_{n-1}] \\
0 & 1
\end{array}\right)w,1\right],
\end{align*} 
where $0\leq k,s\leq q-1$. For $0\leq l\leq f-1$ and $m\geq 1$, we define the following sets
\begin{align*}
\mathcal{S}_m^l &=\{s_n^{q-1-r+p^l}\}_{n\geq m}\cup\{\beta s_n^{q-1-r+p^l}\}_{n\geq m} \\
\mathcal{S}_m &= \bigcup_{l=0}^{f-1}\mathcal{S}_m^l,\\
\mathcal T_m^l &=\{t_n^{r+p^l}\}_{n\geq m}\cup\{\beta t_n^{r+p^l}\}_{n\geq m},\\
\mathcal T_m &= \bigcup_{l=0}^{f-1}T_m^l.
\end{align*}

Now we state the main theorem of this paper.

\begin{theorem}\label{thm-main}
Let $F$ be a finite extension of $\mathbb{Q}_p$ with ramification index $e$. Let $\mathbb{F}_q$ be the residue field of $F$ with $q=p^f.$ Let $0<r<q-1$ and $r=r_0+r_1p+\dots+r_{f-1}p^{f-1}$ with $0<r_j<p-1$ for all $0\leq j\leq f-1$. When $f=1$, we assume $2<r<p-3$. Then a basis of the space of $I(1)$-invariants of the representation
\[\pi_r=\frac{{\rm ind}_{IZ}^G \chi_r}{(\Ker~T_{-1,0},\Ker~T_{1,2})}\]
as an $\overline{\mathbb F}_p$-vector space is given by the images of the following sets in $\pi_r$:
\begin{enumerate}
\item[{\rm(1)}] $\left\{\left[{\rm Id}, 1\right], \left[\beta, 1\right]\right\}$ when $F=\mathbb Q_p$ 
\item[{\rm(2)}] $\mathcal{S}_2\bigcup \left\{\left[{\rm Id},1\right],\left[\beta,1\right]\right\}\bigcup \mathcal T_2~$ when $F \neq \mathbb Q_p$.
\end{enumerate}
\end{theorem}

\begin{remark}\label{new}
The representation $\pi_r$ that we construct and investigate in this paper is a quotient of the representation $\tau_r$ considered in \cite{bl94,bre03,sch11,hen18};
\[0 \rightarrow \frac{\Ker~T_{-1,0}}{{\rm Im~}T_{1,2}} \rightarrow \tau_r \rightarrow \pi_r \rightarrow 0.\]
When $F$ is totally ramified over $\mathbb Q_p$, the representations $\tau_r$ and $\pi_r$ are isomorphic by Theorem \ref{iso} together with the equality of spaces in (\ref{imker}). However, $\pi_r$ is a ``new" representation when $F$ is not totally ramified over $\mathbb Q_p$. That there is no isomorphism between $\tau_r$ and $\pi_r$ can be checked, for instance, from the characterization of the space of $I(1)$-invariants of $\pi_r$ in Theorem \ref{thm-main} and that of $\tau_r$ in \cite[Theorem 1.2]{hen18}. We give more details in \S \ref{isodetails}.
\end{remark}

Following the argument in \cite[Conclusion 3.10]{hen18} word to word, we get the following corollary to Theorem \ref{thm-main}.

\begin{corollary} 
The representation $\pi_r$ is indecomposable; i.e.,
${\rm End}_G(\pi_r) \simeq \overline{\mathbb F}_p.$
\end{corollary}

The plan of the paper is as follows. We collect many results about the Iwahori-Hecke operators in Section \ref{prelim}. Several of these results are contained in some form in \cite{ab13,ab15}. Theorem \ref{thm-main} and the key ideas in its proof are inspired by the work of Hendel \cite{hen18}, though the Iwahori-Hecke approach which is employed in this paper as in \cite{ab13,ab15} seems to be more amenable to carrying out the necessary calculations. We take up the proof in Section \ref{I1}. 

\section{Two basic results}\label{basic}

As in the work of Hendel \cite{hen18}, we will need to frequently make use of the following two results in our computations. 

The first one is the classical result in modular combinatorics due to Lucas which gives a condition for a binomial coefficient ${n \choose r}$ to be zero modulo $p$. 
\begin{theorem}[Lucas]\label{lucas}
Let $n,r\in \mathbb{N}$ be such that $n=\sum\limits_{\substack{i=0}}^kn_ip^i$ and $r=\sum\limits_{\substack{i=0}}^kr_ip^i,$ where $0\leq n_i\leq p-1$ and $0\leq r_i\leq p-1.$ Then \[{n\choose r}\equiv\prod_{i=0}^k{n_i\choose r_i}\mod p.\] 
\end{theorem}

\begin{corollary}\label{prime divides bio coeff}
Let $n,r\in \mathbb{N}.$ Then $p$ divides ${n\choose r}$ if and only if $n_i<r_i$ for some $0\leq i\leq k.$
\end{corollary}

The next result gives a formula for adding multiplicative representatives in $\mathcal O$ \rm \cite[Lemma 1.7]{hen18}. As in \cite{hen18}, this formula will play a crucial role in the calculations to follow.
\begin{lemma}\label{witt vectors sum}
Let $x, y \in \mathbb{F}_q$ with $q=p^f$. Then 
\[[x]+[y]\equiv [x+y]+\varpi^e[P_0(x,y)] \mod \varpi^{e+1},\] 
where $P_0(x,y)=\frac{x^{q^e}+y^{q^e}-(x+y)^{q^e}}{\varpi^e}.$
\end{lemma}

\section{Preliminaries on the Iwahori-Hecke operators}\label{prelim}

For $n\in\mathbb{N} \cup \{0\}$ and $\lambda\in I_n$, define
\[g_{n,\lambda}^0=\left(\begin{array}{cc}
\varpi^n & \lambda \\
0 & 1
\end{array}\right)~~\&~~ g_{n,\lambda}^1=\left(\begin{array}{cc}
1             & 0 \\
\varpi\lambda & \varpi^{n+1}
\end{array}\right).\]
We have the relations 
\[g_{0,0}^0={\rm Id},~g_{0,0}^1=\alpha, \beta g_{n,\lambda}^0=g_{n,\lambda}^1w.\]
Now $G$ acts transitively on the Bruhat-Tits tree of ${\rm SL}_2(F)$, whose vertices are in a $G$-equivariant bijection with the cosets $G/{KZ}$ and whose oriented edges are in a $G$-equivariant bijection with the cosets $G/{IZ}$. We have the explicit Cartan decomposition given by
\[G=\underset{\substack{i \in \{0, 1\} \\ n \geq 0,~ \lambda \in I_n}}\coprod g_{n,\lambda}^i KZ \] 
and an explicit set of coset representatives of $G/IZ$ is given by
\begin{equation}\label{edges}
\left\{  g_{n,\lambda}^0, g_{n,\lambda}^0 
\left(\begin{array}{cc}
1& \mu \\
0 & 1
\end{array}\right)w, g_{n,\lambda}^1w, g_{n,\lambda}^1w 
\left(\begin{array}{cc}
1& \mu \\
0 & 1
\end{array}\right)w
\right\}_{n \geq 0, \lambda \in I_n},
\end{equation}
where $\mu\in I_1.$ 

Now we recall a few details about the Iwahori-Hecke algebra \cite[\S 3.2]{bl94}. By definition, this algebra, denoted by $\mathcal{H}(IZ,\chi_r)$, is the endomorphism algebra of the compactly induced representation ${\rm ind}_{IZ}^G \chi_r$. For $n \in \mathbb Z$, let $\phi_{n, n+1}$ denote the convolution map supported on $IZ{\alpha}^{-n}I$ such that $\phi_{n, n+1}(\alpha^{-n})=1$ \cite[Lemma 9]{bl94}. We denote by $T_{n, n+1}$ the corresponding element in $\mathcal{H}(IZ,\chi_r)$. By \cite[Proposition 13]{bl94}, for $0<r<q-1$, we have:
\[\mathcal{H}(IZ,\chi_r) \simeq\frac{\overline{\mathbb{F}}_p[T_{-1,0,},T_{1,2}]}{(T_{-1,0}T_{1,2},T_{1,2}T_{-1,0})}.\]
Substituting $n=1$ in \cite[(16), (17)]{bl94}, we have the following explicit formulas for $T_{-1,0}$ and $T_{1,2}$: 
\begin{equation}\label{formula T_{-1,0}}
T_{-1,0}(\left[g,1\right])=\sum\limits_{\substack{\lambda\in I_1}}\left[gg_{1,\lambda}^0,1\right],
\end{equation}
\begin{equation}\label{formula T_{1,2}}
T_{1,2}(\left[g,1\right])=\sum\limits_{\substack{\lambda\in I_1}}\left[g\beta\left(\begin{array}{cc}
1& \lambda \\
0 & 1
\end{array}\right)w,1\right].
\end{equation} 

The following proposition characterizes the kernel of the Iwahori-Hecke operators $T_{-1,0}$ and $T_{1,2}$ \cite{ab13,ab15}.
\begin{proposition}\label{kernels}
We have:
\begin{enumerate}
\item ${\rm Ker~}T_{-1,0}$ is generated as a $G$-module by the vectors 
\begin{enumerate}
\item $(-1)^{q-1-r}s_0^0+t_1^r$, 
\item $t_1^s$ where $0\leq s \leq r-1$, 
\item $t_1^s$ where $s > r$ and ${q-1-r \choose q-1-s} \equiv 0 \mod p$.
\end{enumerate}
\item ${\rm Ker~}T_{1,2}$ is generated as a $G$-module by the vectors 
\begin{enumerate}
\item $t_0^0+s_1^{q-1-r}$, 
\item $s_1^k$ where $0\leq k \leq q-2-r$, 
\item $s_1^k$ where $k > q-1-r$ and ${r \choose q-1-k} \equiv 0 \mod p$.
\end{enumerate}
\end{enumerate}
\end{proposition}
\begin{proof}
We indicate the proof for Ker $T_{1,2}$, with the other case being similar. An arbitrary vector in ind$_{IZ}^G \chi_r$ is an $\overline{\mathbb F}_p$-linear combination of vectors $[g,1]$,  where $g$ is in the set of coset representatives (\ref{edges}) of $G/IZ$. Arguing as in the proof of \cite[Proposition 3.1]{ab15}, we can restrict our attention to the vectors
\[\left\{[{\rm Id},1],[\beta,1], [g_{1,\mu}^0,1], \left[
\left(\begin{array}{cc}
1& \mu \\
0 & 1
\end{array}\right)w,1\right] \right\}\]
for $\mu \in I_1$. Now the proof boils down to elementary linear algebra as in \cite[Lemma 3.2]{ab15}, where one is led to analyse the indices $i$ for which 
\[\sum_{\mu \in \mathbb F_q} \mu^i(\mu-\lambda)^r = 0,\]
for $\lambda \in \mathbb F_q$. Alternatively, this last step can be deduced directly from the explicit formulas for the Iwahori-Hecke operators in \cite[p. 63-64]{ab13}. 
\end{proof}

\begin{remark}\label{rmk-imker2}
We remarked in (\ref{imker2}) in Section \ref{introduction} that we have strict containments
\begin{equation}
{\rm Im}~ T_{-1,0} \subsetneq {\rm Ker}~ T_{1,2} ~\&~ {\rm Im}~ T_{1,2} \subsetneq {\rm Ker}~ T_{-1,0}.
\end{equation}
when $F$ is not a totally ramified extension of $\mathbb Q_p$. The reason for this is that the third type of vectors in both (1) and (2) in Proposition \ref{kernels} do not belong to the images of the Iwahori-Hecke operators. Note that such vectors do not exist when $f=1$; i.e., when $q=p$. By the argument in \cite[Lemma 3.2]{ab15}, it can be shown that the first two types of vectors are indeed in the image of the relevant Iwahori-Hecke operator.
\end{remark}

\begin{corollary}\label{invariants}
A basis of the space of $I(1)$-invariants of ${\rm Ker~}T_{-1,0}$ is given by $\{t_n^0,\beta t_n^0\}_{n \geq 1}$ and that of ${\rm Ker~}T_{1,2}$ is given by $\{s_n^0,\beta s_n^0\}_{n \geq 1}$. Moreover, the action of $I$ is given by 
\[\left(\begin{array}{cc}
a & b\\
\varpi c & d
\end{array}\right) \cdot v = 
\begin{cases}
a^r v &\text{$v = t_n^0$ or $\beta s_n^0$,} \\
d^r v &\text{$v = s_n^0$ or $\beta t_n^0$.}
\end{cases}
\]
\end{corollary}

\begin{proof}
The first part of Proposition \ref{kernels} together with the observation that the space of $I(1)$-invariants of the full induced representation is given by
\[\left( {\rm ind}_{IZ}^G \chi_r \right)^{I(1)} = \langle s_n^0,t_n^0,\beta s_n^0,\beta t_n^0\rangle_{n\geq 0}.\]
For the second part, observe that since 
\[I/{I(1)}=\left\{\left(\begin{array}{cc}
a & 0\\
0 & d
\end{array}\right)\mid a,d\in{\mathbb{F}_q}^{\times}\right\},\]  
it follows that 
\[\left(\begin{array}{cc}
a & b\\
\varpi c & d
\end{array}\right) ~\&~
\left(\begin{array}{cc}
a & 0\\
0 & d
\end{array}\right)
\]
have the same action on any $I(1)$-invariant vector. Now, for any $k \geq 0$, we have
\begin{align*}
\left(\begin{array}{cc}
a & 0\\
0 & d
\end{array}\right)s_n^k &=\left(\begin{array}{cc}
a & 0\\
0 & d
\end{array}\right)\sum\limits_{\substack{\mu\in I_n}}\mu_{n-1}^k\left[\left(\begin{array}{cc}
\varpi^n & \mu \\
0    &  1
\end{array}\right),1\right]\\
&=\sum\limits_{\substack{\mu\in I_n}}\mu_{n-1}^k
\left[\left(\begin{array}{cc}
\varpi^n & ad^{-1}\mu\\
0  &  1
\end{array}\right)\left(\begin{array}{cc}
a & 0\\
0 & d
\end{array}\right),1\right]\\
&= d^r(da^{-1})^k s_n^k.   
\end{align*}
A similar computation gives  
\begin{align*}
\left(\begin{array}{cc}
a & 0\\
0 & d
\end{array}\right)t_n^s &=a^r(da^{-1})^st_n^s.
\end{align*}
Similarly, we can check the action on $\beta s_n^k$ and $\beta t_n^k$.
\end{proof}

Next, we recall \cite[Proposition 3.3]{ab15}, whose proof in [loc. cit.] is valid for any $q$.
\begin{proposition}\label{intersection}
We have
\[\Ker~T_{-1,0} \cap \Ker~T_{1,2}=\{0\}.\]
\end{proposition}

As a corollary to Proposition \ref{intersection}, we have the following lemma.
\begin{lemma}\label{kernels sum}
For the iwahori-Hecke operators $T_{-1,0}$ and $T_{1,2}$, we have
\[{\rm ind}_{IZ}^G \chi_r \neq \Ker~T_{-1,0}\oplus\Ker~T_{1,2}.\] 
\end{lemma}
\begin{proof}
If possible, let
\[{\rm ind}_{IZ}^G \chi_r=\Ker~T_{-1,0}\oplus\Ker~T_{1,2}.\] 
Then we get
$[{\rm Id},1]=v_1+v_2$ for some $v_1\in\Ker~T_{-1,0}$ and $v_2\in\Ker~T_{1,2}$.
Then, for an element $g \in I$, 
\[ g(v_1+v_2)= \left(\begin{array}{cc}
a & b\\
\varpi c & d
\end{array}\right) (v_1+v_2) = d^r[{\rm Id},1]=d^r(v_1+v_2)\] 
and this implies
\[gv_1-d^rv_1=-gv_2+d^rv_2 = 0,\]
by Proposition \ref{intersection}. In particular, both $v_1$ and $v_2$ are $I(1)$-invariant. 
By Corollary \ref{invariants}, $v_1$ is a linear combination of vectors of the form $\{\beta t_n^0\}_{n\geq 1}$ and $v_2$ is a linear combination of vectors of the form $\{s_n^0\}_{n\geq 1}$. But $[{\rm Id},1]$ cannot be written as a linear combination of these types of vectors. 
\end{proof}

We end this section with two more results which immediately follow from considerations similar to Proposition \ref{kernels}. We state these in a ready to use format here (see also \cite[p. 63-64]{ab13}).

\begin{lemma}\label{kernel condition}
Let $0\leq i_j\leq q-1$ for $0\leq j\leq n-1$ and $\mu=[\mu_0]+[\mu_1]\varpi+\dots+[\mu_{n-1}]\varpi^{n-1} \in I_n.$ Write $i_{n-1}=i_{n-1,0}+i_{n-1,1}p+\dots+i_{n-1,f-1}p^{f-1}.$ Then
\begin{itemize}
\item[{\rm(1)}] $\sum\limits_{\substack{\mu_0}}\dots\sum\limits_{\substack{\mu_{n-1}}}\mu_0^{i_0}\dots\mu_{n-1}^{i_{n-1}}\left[g^0_{n,\mu}, 1\right]\in \Ker~T_{1,2}$ if and only if $0\leq i_{n-1}\leq q-2-r$ or $i_{n-1}>q-1-r$ such that $i_{n-1,j}<p-1-r_j$ for some $0\leq j\leq f-2$. 
\item[{\rm(2)}] $\sum\limits_{\substack{\mu_0}}\dots\sum\limits_{\substack{\mu_{n-1}}}\mu_0^{i_0}\dots\mu_{n-1}^{i_{n-1}}\left[g^0_{n-1,[\mu]_{n-1}}\left(\begin{array}{cc}
1& [\mu_{n-1}]\\
0& 1
\end{array}\right)w,1\right]\in \Ker~T_{-1,0}$ if and only if 
$0\leq i_{n-1}\leq r-1$ or $i_{n-1}>r$ such that $i_{n-1,j}<r_j$ for some $0\leq j\leq f-2$. 
\end{itemize}
\end{lemma}

\begin{remark}\label{rmk-f-1}
Note that in Lemma \ref{kernel condition}, the range for $j$ is $0 \leq j \leq f-2$ because 
\[i_{n-1} > q-1-r \implies i_{n-1,f-1} \geq p-1-r_{f-1}.\]
\end{remark}

\begin{remark}
Note that the condition
 \[i_{n-1}>q-1-r ~\&~  i_{n-1,j}<p-1-r_j \mbox{~for~some~} 0\leq j\leq f-2\]
 in Lemma \ref{kernel condition} (1) is precisely what gives, by Theorem \ref{lucas},
 \[{r \choose q-1-i_{n-1}} \equiv 0 \mod p\]
 which is related to the condition in (c) of Proposition \ref{kernels} (2). Similarly, the condition
 \[i_{n-1}>r ~\&~ i_{n-1,j}<r_j \mbox{~for~some~} 0\leq j\leq f-2\]
in Lemma \ref{kernel condition} (2) is related to (c) of Proposition \ref{kernels} (1).
\end{remark}

The following lemma is \cite[Lemma 3.1]{ab13}. We note that its proof in [loc. cit.] is valid for any $q$.

\begin{lemma}\label{modulo kk}
Let $\mu = [\mu_0]+\dots+[\mu_{n-1}]\varpi^{n-1} \in I_n.$ Then modulo $(\Ker~ T_{-1,0},\Ker~T_{1,2})$, we have the identities
\begin{enumerate}
\item[{\rm(1)}]$\sum\limits_{\substack{\mu_{n-1}\in I_1}}\mu_{n-1}^{q-1-r}\left[g^0_{n,\mu},1\right]=
-\left[g^0_{n-2,[\mu]_{n-2}}\left(\begin{array}{cc}
1& [\mu_{n-2}]\\
0& 1
\end{array}\right)w,1\right],$
\item[{\rm(2)}]$\sum\limits_{\substack{\mu_{n-1}\in I_1}}\mu_{n-1}^r\left[g^0_{n-1,[\mu]_{n-1}}\left(\begin{array}{cc}1& [\mu_{n-1}]\\
0& 1
\end{array}\right)w,1\right]=(-1)^{r-1}\left[g^0_{n-1,[\mu]_{n-1}},1\right].$
\end{enumerate}
\end{lemma}

\begin{remark}
In fact, (1) is true modulo $\Ker~T_{1,2}$ and (2) is true modulo $\Ker~T_{-1,0}$ (cf. \cite[(4) \& (5) on p. 62]{ab13}).
\end{remark}

\section{Proof of Theorem \ref{thm-main}}\label{I1}

In this section we take up the proof of Theorem \ref{thm-main}. As mentioned in Section \ref{introduction}, several of the ideas of the proof here are already there in \cite{hen18}.

\subsection{A set of $I(1)$-invariants}

First we make the following observation \cite[\S 2.1]{hen18}. For $a, b, c \in \mathcal{O}$, any matrix in $I(1)$ can be written as
\[\left(\begin{array}{cc}
1+\varpi a & b \\
\varpi c   & 1+\varpi d
\end{array}\right)=\left(\begin{array}{cc}
1 & (1+\varpi d)^{-1}b \\
0 & 1
\end{array}\right)\left(\begin{array}{cc}
1 & 0 \\
\varpi c t^{-1} & 1
\end{array}\right)\left(\begin{array}{cc}
t & 0 \\
0 & 1+\varpi d
\end{array}\right),\]
where $t=1+\varpi(a-bc(1+\varpi d)^{-1}).$ Hence to prove that a certain vector is $I(1)$-invariant modulo $(\Ker~T_{-1,0}, \Ker~T_{1,2})$, it is enough to check for invariance under
\[\left(\begin{array}{cc}
1 & b\\
0 & 1
\end{array}\right), \left(\begin{array}{cc}
1     & 0\\
\varpi c &  1
\end{array}\right), \left(\begin{array}{cc}
1+\varpi a & 0\\
0       &  1
\end{array}\right),\] where $a,b,c\in\mathcal{O}.$

We first prove that the set of vectors $\mathcal S_2$ and $\mathcal T_2$ are $I(1)$-invariants when considered as vectors in $\pi_r$; i.e., when we consider the images of these vectors modulo $\Ker~T_{-1,0}\oplus \Ker~T_{1,2}$. The first step in achieving this is an inductive argument which reduces the general case to the case $n=2$.

\begin{lemma}\label{enough to look at ball one for kk}
If $s_{n-1}^k$ (resp. $t_{n-1}^s$) is $I(1)$-invariant modulo $(\Ker~T_{-1,0},\Ker~T_{1,2})$, then, for all $n\geq 2$, the vector $s_n^k$ (resp. $t_n^s$) is also $I(1)$-invariant modulo $(\Ker~T_{-1,0},\Ker~T_{1,2})$.
\end{lemma}  

\begin{proof}
We prove the case of $s_n^k$ and the case of $t_n^s$ is similar. Assume that $s_{n-1}^k$ is $I(1)$-invariant modulo $(\Ker~T_{-1,0},\Ker~T_{1,2}).$ 

Now,
\begin{align*}
& \left(\begin{array}{cc}
1 & b\\
0 & 1
\end{array}\right)s_n^k \\
&=\sum\limits_{\substack{\mu\in I_n}}\mu_{n-1}^k
\left[\left(\begin{array}{cc}
1 & b\\
0 & 1
\end{array}\right)
\left(\begin{array}{cc}
\varpi^n & \mu\\
0 & 1
\end{array}\right),1\right] \\
&=\sum\limits_{\substack{\mu\in I_n}}\mu_{n-1}^k
\left[\left(\begin{array}{cc}
1 & b\\
0 & 1
\end{array}\right)
\left(\begin{array}{cc}
\varpi & [\mu_0]\\
0 & 1
\end{array}\right)
\left(\begin{array}{cc}
\varpi^{n-1} & \displaystyle{\sum_{i=1}^{n-1}} [\mu_i]\varpi^{i-1}\\
0     & 1
\end{array}\right),1\right]\\
&=\sum\limits_{\substack{\mu\in I_n}}\mu_{n-1}^k
\left[\left(\begin{array}{cc}
\varpi & [\mu_0+b_0]\\
0 & 1
\end{array}\right)
\left(\begin{array}{cc}
1 & B(\mu_0,b)\\
0 & 1
\end{array}\right)
\left(\begin{array}{cc}
\varpi^{n-1} & \displaystyle{\sum_{i=1}^{n-1}} [\mu_i]\varpi^{i-1}, \\
0     & 1
\end{array}\right),1\right]
\end{align*}
where 
\begin{equation}\label{bmu0}
B(\mu_0, b)=\varpi^{e-1}[P_0(\mu_0,b_0)]+[b_1]+[b_2]\varpi+\dots
\end{equation}
and  
\begin{equation}\label{p0mu0}
P_0(\mu_0, b_0)=\frac{\mu_0^{q^e}+b_0^{q^e}-(\mu_0+b_0)^{q^e}}{\varpi^e}
\end{equation}
is obtained from the formula in Lemma \ref{witt vectors sum}. Let
\[\mu'=[\mu_1]+[\mu_2]\varpi+\dots+[\mu_{n-1}]\varpi^{n-2}.\]
We continue by making the substitution $\mu_0 \rightarrow \mu_0-b_0$. Thus,
\begin{align*}
& \left(\begin{array}{cc}
1 & b\\
0 & 1
\end{array}\right)s_n^k \\
&=\sum\limits_{\substack{\mu_0\in I_1}}
\left(\begin{array}{cc}
\varpi & [\mu_0]\\
0 & 1
\end{array}\right)
\left\{\left(\begin{array}{cc}
1 & B(\mu_0-b_0,b)\\
0 & 1
\end{array}\right)
\sum\limits_{\substack{\mu'\in I_{n-1}}}{\mu'}_{n-2}^k
\left[\left(\begin{array}{cc}
\varpi^{n-1} & \mu'\\
0     & 1
\end{array}\right),1\right]\right\}\\
&=\sum\limits_{\substack{\mu_0\in I_1}}
\left(\begin{array}{cc}
\varpi & [\mu_0]\\
0 & 1
\end{array}\right)
\left\{\sum\limits_{\substack{\mu'\in I_{n-1}}}{\mu'}_{n-2}^k\left[\left(\begin{array}{cc}
\varpi^{n-1} & \mu'\\
0     & 1
\end{array}\right),1\right]+x_{\mu_0}\right\},
\end{align*}
by our assumption, where $x_{\mu_0}\in (\Ker~T_{-1,0}, \Ker~T_{1,2})$. Thus, we get
\begin{align*}
\left(\begin{array}{cc}
1 & b\\
0 & 1
\end{array}\right)s_n^k
&=s_n^k+\sum\limits_{\substack{\mu_0\in I_1}}
\left(\begin{array}{cc}
\varpi & [\mu_0]\\
0 & 1
\end{array}\right)x_{\mu_0},
\end{align*}
and hence 
\[\left(\begin{array}{cc}
1 & b\\
0 & 1
\end{array}\right)s_n^k-s_n^k\in (\Ker~T_{-1,0}, \Ker~T_{1,2}).\]
Checking for invariance under
\[\left(\begin{array}{cc}
1     & 0\\
\varpi c &  1
\end{array}\right) ~\&~ \left(\begin{array}{cc}
1+\varpi a & 0\\
0       &  1
\end{array}\right)\]
is even easier which we skip. 
\end{proof}

Now we take the case $n=2$. Recall that $e$ (resp. $f$) is the ramification index (resp. residue degree) of $F$ over $\mathbb Q_p$. We write 
\[r=r_0+r_1p+\dots+r_{f-1}p^{f-1}\]
where $0 \leq r_j \leq p-1$ for $0 \leq j \leq f-1$.

We first observe that for $a,b,c \in \mathcal{O}$, we have
\begin{equation}\label{eqn-k}
\left(\begin{array}{cc}
1+\varpi a & b\\
c\varpi   & 1+d\varpi
\end{array}\right) \left(\begin{array}{cc}
\varpi & [\mu]\\
0 &  1
\end{array}\right) = \left(\begin{array}{cc}
\varpi & [\mu+b_0]\\
0 &  1
\end{array}\right)k,
\end{equation}
for $k \in I(1)$. Indeed,
\begin{align*}
{\rm LHS} &= \left(\begin{array}{cc}
\varpi(1+\varpi a) & [\mu+b_0]+\varpi(*)\\
c\varpi^2   &  1+\varpi(\Delta)
\end{array}\right)\\
&= \left(\begin{array}{cc}
\varpi & [\mu+b_0]\\
0 &  1
\end{array}\right) \left(\begin{array}{cc}
1+a\varpi-[\mu+b_0]c\varpi & (*)-(\mu+b_0)\Delta\\
c\varpi^2        &  1+\varpi\Delta
\end{array}\right),
\end{align*}
where $*,\Delta\in \mathcal{O}$. 
Similarly, one can show that 
\begin{equation}\label{eqn-k'}
\left(\begin{array}{cc}
1+\varpi a & b\\
c\varpi   & 1+d\varpi
\end{array}\right) \left(\begin{array}{cc}
1 & [\mu]\\
0 &  1
\end{array}\right)w= \left(\begin{array}{cc}
1 & [\mu+b_0]\\
0 &  1
\end{array}\right)wk^\prime 
\end{equation} 
for some $k^\prime \in I(1)$. 

\begin{lemma}\label{ball two invariants} 
Assume $0 < r_j < p-1$, and if $f=1$, assume further that $2 < r < p-3$. Then when $(e,f) \neq (1,1)$, we have 
\[gs_2^{q-1-r+p^l}-s_2^{q-1-r+p^l} \in (\Ker~T_{-1,0},\Ker~T_{1,2})\] 
and
\[gt_2^{r+p^l}-t_2^{r+p^l} \in (\Ker~T_{-1,0},\Ker~T_{1,2})\]
for all $g\in I(1)$ and $0\leq l\leq f-1.$ 
\end{lemma}

\begin{proof}
We have
\begin{align*}
& \left(\begin{array}{cc}
1 & b\\
0 & 1
\end{array}\right)s_2^{q-1-r+p^l} \\
&= \sum\limits_{\substack{\mu\in I_2}}\mu_1^{q-1-r+p^l}
\left[\left(\begin{array}{cc}
1 & b\\
0 & 1
\end{array}\right)\left(\begin{array}{cc}
\varpi^2 & [\mu_0]+[\mu_1]\varpi\\
0 &  1
\end{array}\right),1\right] \\
&= \sum\limits_{\substack{\mu\in I_2}}\mu_1^{q-1-r+p^l}
\left[\left(\begin{array}{cc}
1 & b\\
0 & 1
\end{array}\right)\left(\begin{array}{cc}
\varpi & [\mu_0]\\
0 &  1
\end{array}\right)\left(\begin{array}{cc}
\varpi & [\mu_1]\\
0 &  1
\end{array}\right),1\right] \\
&= \sum\limits_{\substack{\mu\in I_2}}\mu_1^{q-1-r+p^l}
\left[\left(\begin{array}{cc}
\varpi & [\mu_0+b_0]\\
0 &  1
\end{array}\right)\left(\begin{array}{cc}
1 & B(\mu_0,b)\\
0 &  1
\end{array}\right)
\left(\begin{array}{cc}
\varpi & [\mu_1]\\
0 &  1
\end{array}\right),1\right],
\end{align*}
where $B(\mu_0,b)$ is given by (\ref{bmu0}) in the proof of Lemma \ref{enough to look at ball one for kk}. Now write 
\[B(\mu_0, b)=[b_1+Z]+(*)\varpi\]
where $Z=0$ for $e>1$ and $Z=P_0(\mu_0,b_0)$ for $e=1$. 

To continue, the above expression equals
\[\sum\limits_{\substack{\mu\in I_2}}\mu_1^{q-1-r+p^l}
\left[\left(\begin{array}{cc}
\varpi & [\mu_0+b_0]\\
0 &  1
\end{array}\right)\left(\begin{array}{cc}
1 & [b_1+Z]+(*)\varpi\\
0 &  1
\end{array}\right)
\left(\begin{array}{cc}
\varpi & [\mu_1]\\
0 &  1
\end{array}\right),1\right]\]
which equals
\[\sum\limits_{\substack{\mu\in I_2}}\mu_1^{q-1-r+p^l}
\left[\left(\begin{array}{cc}
\varpi & [\mu_0+b_0]\\
0 &  1
\end{array}\right)\left(\begin{array}{cc}
\varpi & [\mu_1+b_1+Z]\\
0 &  1
\end{array}\right)k,1\right] \]
for $k \in I(1)$, by (\ref{eqn-k}). We continue by making the change of variables
\[\mu_1 \rightarrow \mu_1-b_1-Z ~\&~ \mu_0 \rightarrow \mu_0-b_0,\] 
and we get
\begin{align*}
& \sum\limits_{\substack{\mu\in I_2}}(\mu_1-b_1-Z)^{q-1-r+p^l}
\left[\left(\begin{array}{cc}
\varpi^2 & \mu\\
0   &  1
\end{array}\right),1\right]\\
& = s_2^{q-1-r+p^l}+\sum\limits_{\substack{\mu\in I_2}}\sum\limits_{\substack{i=0}}^{q-1-r+p^l-1}{q-1-r+p^l\choose i}(-b_1-Z)^{q-1-r+p^l-i}\mu_1^i
\left[\left(\begin{array}{cc}
\varpi^2 & \mu\\
0   &  1
\end{array}\right),1\right].
\end{align*}

Now we read the above expression modulo $\Ker~T_{1,2}$. We claim that only the term corresponding to $i=q-1-r$ remains amongst the $q-1-r+p^l$ terms in the inner summation in the above expression. By Lemma \ref{modulo kk} (1), we know that
\[\sum_{\mu\in I_2} \mu_1^i
\left[\left(\begin{array}{cc}
\varpi^2 & \mu\\
0   &  1
\end{array}\right),1\right] \in \Ker~T_{1,2}\] 
precisely when $0 \leq i \leq q-2-r$ or $i > q-1-r$ such that $i_j < p-1-r_j$ for some $0 \leq j \leq f-2$. Note that if $i > q-1-r$ and $i_j \geq p-1-r_j$ for all $0 \leq j \leq f-1$ (cf. Remark \ref{rmk-f-1}) then $i_j > p-1-r_j$ for some $0 \leq j \leq l-1$ (since $i \leq q-1-r+p^l-1$). If this is the case then observe that 
\[ {q-1-r+p^l \choose i} \equiv 0 \mod p  \]
by Corollary \ref{prime divides bio coeff}. Thus, modulo $\Ker~T_{1,2}$, we get
\begin{align*}
& \left(\begin{array}{cc}
1 & b\\
0 & 1
\end{array}\right)s_2^{q-1-r+p^l}  \\
&=s_2^{q-1-r+p^l}+\sum\limits_{\substack{\mu\in I_2}}{q-1-r+p^l\choose q-1-r}(-b_1-Z)^{p^l}\mu_1^{q-1-r}
\left[\left(\begin{array}{cc}
\varpi^2 & \mu\\
0   &  1
\end{array}\right),1\right] \\
&= s_2^{q-1-r+p^l}+\sum\limits_{\substack{\mu_0\in I_1}}(p-r_l)(-b_1-Z)^{p^l}
\left[\left(\begin{array}{cc}
1 & [\mu_0]\\
0 &  1
\end{array}\right)w,1\right]
\end{align*}
by Lemma \ref{modulo kk} (1) and the binomial coefficient here is computed via Theorem \ref{lucas}.

Now if $e>1$ then we have $Z=0$. Therefore, it follows, by Lemma \ref{kernel condition} (2), that
\[\left(\begin{array}{cc}
1 & b\\
0 & 1
\end{array}\right)s_2^{q-1-r+p^l}-s_2^{q-1-r+p^l}\in \Ker~T_{-1,0},\]
and thus we have proved
\[ \left(\begin{array}{cc}
1 & b\\
0 & 1
\end{array}\right)s_2^{q-1-r+p^l} \equiv s_2^{q-1-r+p^l} \mod (\Ker~T_{-1,0},\Ker~T_{1,2}). \]

If $e=1$ then $Z=P_0(\mu_0,b_0)$. As $F$ is unramified over $\mathbb Q_p$, we have $\varpi = p$. Now by Corollary \ref{prime divides bio coeff}, it follows that
\begin{align*}
Z &= \frac{\mu_0^{q^e}+b_0^{q^e}-(\mu_0+b_0)^{q^e}}{\varpi^e} \\
&\equiv -\sum\limits_{\substack{i=1}}^{p-1}\frac{1}{p}{p^f\choose ip^{f-1}}b_0^{p^f-ip^{f-1}}\mu_0^{ip^{f-1}}\mod p.
\end{align*} 
In this case, if further $f\neq 1$ we have, modulo $\Ker~T_{1,2}$, 
\begin{align*}
\left(\begin{array}{cc}
1 & b\\
0 & 1
\end{array}\right)s_2^{q-1-r+p^l}-s_2^{q-1-r+p^l} &= \sum\limits_{\substack{\mu_0\in I_1}}(p-r_l)(-b_1-Z)^{p^l}
\left[\left(\begin{array}{cc}
1   & [\mu_0]\\
0   &  1
\end{array}\right)w,1\right] \\
&= \sum\limits_{\substack{\mu_0\in I_1}}r_l(b_1^{p^l}+Z^{p^l})
\left[\left(\begin{array}{cc}
1   & [\mu_0]\\
0   &  1
\end{array}\right)w,1\right].
\end{align*} 
Note that both
\[ \sum_{\mu_0 \in I_1} \left[\left(\begin{array}{cc}
1   & [\mu_0]\\
0   &  1
\end{array}\right)w,1\right] ~\&~ \sum_{\mu_0 \in I_1} \mu_0^{ip^{l-1}} \left[\left(\begin{array}{cc}
1   & [\mu_0]\\
0   &  1
\end{array}\right)w,1\right]\]
are in $\Ker~T_{-1,0}$, by Lemma \ref{kernel condition} (2). Thus, once again we have proved
\[ \left(\begin{array}{cc}
1 & b\\
0 & 1
\end{array}\right)s_2^{q-1-r+p^l} \equiv s_2^{q-1-r+p^l} \mod (\Ker~T_{-1,0},\Ker~T_{1,2}). \]

Now we analyze invariance for the lower unipotent representative of $I(1)$. We have
\begin{align*}
\left(\begin{array}{cc}
1     & 0\\
\varpi c & 1
\end{array}\right)s_2^{q-1-r+p^l} &= \sum\limits_{\substack{\mu\in I_2}}\mu_1^{q-1-r+p^l}\left[\left(\begin{array}{cc}
1     & 0\\
\varpi c & 1
\end{array}\right)\left(\begin{array}{cc}
\varpi^2 & [\mu_0]+[\mu_1]\varpi\\
0 &  1
\end{array}\right), 1\right] \\
&= \sum\limits_{\substack{\mu\in I_2}}\mu_1^{q-1-r+p^l}
\left[\left(\begin{array}{cc}
1     & 0\\
\varpi c & 1
\end{array}\right)\left(\begin{array}{cc}
\varpi & [\mu_0]\\
0 &  1
\end{array}\right)\left(\begin{array}{cc}
\varpi & [\mu_1]\\
0 &  1
\end{array}\right),1\right]
\end{align*}
which we express as
\[\sum\limits_{\substack{\mu\in I_2}}\mu_1^{q-1-r+p^l}
\left[\left(\begin{array}{cc}
\varpi & [\mu_0]\\
0 &  1
\end{array}\right)\left(\begin{array}{cc}
1-\varpi c[\mu_0] & -[\mu_0^2]c\\
\varpi^2c     &  1+\varpi c[\mu_0]
\end{array}\right)
\left(\begin{array}{cc}
\varpi & [\mu_1]\\
0 &  1
\end{array}\right),1\right]
\]
and this equals
\[
\sum\limits_{\substack{\mu\in I_2}}\mu_1^{q-1-r+p^l}
\left[\left(\begin{array}{cc}
\varpi & [\mu_0]\\
0 &  1
\end{array}\right)\left(\begin{array}{cc}
\varpi & [\mu_1-c_0\mu_0^2]\\
0 &  1
\end{array}\right)k,1\right]
\]
for $k \in I(1)$ by (\ref{eqn-k}). Changing $\mu_1 \rightarrow \mu_1+c_0\mu_0^2$, we get
\begin{align*}
\left(\begin{array}{cc}
1     & 0\\
\varpi c & 1
\end{array}\right)s_2^{q-1-r+p^l}  &= \sum\limits_{\substack{\mu\in I_2}}(\mu_1+c_0\mu_0^2)^{q-1-r+p^l}
\left[\left(\begin{array}{cc}
\varpi^2 & \mu\\
0 &  1
\end{array}\right),1\right] 
\end{align*}
which we read modulo $\Ker~T_{1,2}$ and get 
\begin{align*}
s_2^{q-1-r+p^l}+\sum\limits_{\substack{\mu\in I_2}}{q-1-r+p^l\choose q-1-r}(c_0\mu_0^2)^{p^l}\mu_1^{q-1-r}
\left[\left(\begin{array}{cc}
\varpi^2 & \mu\\
0 &  1
\end{array}\right),1\right]
\end{align*}
 by Corollary \ref{prime divides bio coeff} together with Lemma \ref{kernel condition} (1), exactly as we have argued before. Now this equals, modulo $\Ker~T_{1,2}$,
\begin{align*}
s_2^{q-1-r+p^l}+\sum\limits_{\substack{\mu_0\in I_1}}(p-r_l)c_0^{p^l}\mu_0^{2p^l}
\left[\left(\begin{array}{cc}
1 & [\mu_0]\\
0 &  1
\end{array}\right)w,1\right]
\end{align*}
by Theorem \ref{lucas} and Lemma \ref{modulo kk} (1). By Lemma \ref{kernel condition} (2), this vector belongs to $\Ker~T_{-1,0}$ (with the extra assumption that $3 \leq r$ when $f=1$). Thus, we have proved    
\[\left(\begin{array}{cc}
1     & 0\\
\varpi c & 1
\end{array}\right)s_2^{q-1-r+p^l} \equiv s_2^{q-1-r+p^l} \mod (\Ker ~T_{-1,0},\Ker ~T_{1,2}).\]
 
The proof for showing that
\[\left(\begin{array}{cc}
1+\varpi a  & 0\\
0       & 1
\end{array}\right)s_2^{q-1-r+p^l}-s_2^{q-1-r+p^l}\in (\Ker~T_{-1,0},\Ker~T_{1,2})\]
is similar and therefore we skip it. 

The argument for 
\begin{align*}
gt_2^{r+p^l}-t_2^{r+p^l}\in (\Ker~T_{-1,0},\Ker~T_{1,2})
\end{align*}
for all $g \in I(1)$ is similar to the one for $s_2^{q-1-r+p^l}$. Note that corresponding to the case $3 \leq r$ in the totally ramified case for $s_2^{q-1-r+p^l}$, in the case of $t_2^{r+p^l}$ we will get $r \leq p-4$.
\end{proof}

\subsection{Linear independence}

The following lemma gives the action of the Iwahori subgroup $I$ on the $I(1)$-invariant vectors (cf. \cite[Lemma 3.6]{hen18}).

\begin{lemma}\label{i action}
Let $\left(\begin{array}{cc}
a & b\\
c & d
\end{array}\right)\in I.$ Let $s_n^k$ and $t_n^s$ be $I(1)$-invariants modulo $(\Ker~T_{-1,0},\Ker~T_{1,2}).$ Then they are $I$-eigenvectors and those actions are given by
\begin{itemize}
\item[{\rm(1)}] $\left(\begin{array}{cc}
a & b\\
c & d
\end{array}\right)\cdot s_n^k=d^r(da^{-1})^ks_n^k,$
\item[{\rm(2)}] $\left(\begin{array}{cc}
a & b\\
c & d
\end{array}\right)\cdot t_n^s=a^r(da^{-1})^st_n^s.$ 
\end{itemize} 
\end{lemma}

\begin{proof}
The proof is straightforward and we have already done it in the proof of the second part of Corollary \ref{invariants}.
\end{proof}

\begin{remark}\label{same computation}
Lemmas \ref{enough to look at ball one for kk}, \ref{ball two invariants} and \ref{i action} remain true for $\beta s_n^k$ and $\beta t_n^s$. 
\end{remark}

\begin{proposition}\label{linear independence}
The set of vectors in $\mathcal{S}_2 \cup \mathcal T_2$ of Theorem \ref{thm-main} are linearly independent.
\end{proposition}

\begin{proof}
Note that the vectors in $\mathcal{S}_2 \cup \mathcal T_2$ consist of vectors of the form 
\[s_n^{q-1-r+p^l}, \beta s_n^{q-1-r+p^l}, t_n^{r+p^l}, \beta t_n^{r+p^l} \] for $n\geq 2$ and $0\leq l\leq f-1$. These are invariant under $I(1)$ modulo $(\Ker~T_{-1,0}, \Ker~T_{1,2})$ except for the case when both $e=1$ and $f=1$ (cf. Lemmas \ref{enough to look at ball one for kk}, \ref{ball two invariants} and Remark \ref{same computation}). 

For any vector $v \in {\rm ind}_{IZ}^G \chi_r$, note that $v$ and $\beta v$ cannot cancel each other (pictorially they are on two different sides of the tree of ${\rm SL}_2(F)$). Therefore, it is enough to show that the set $\{s_n^{q-1-r+p^l}, t_n^{r+p^l}\}$, for $n\geq 2$ and $0\leq l\leq f-1$, is linearly independent. Since $s_n^{q-1-r+p^l}$ and $t_n^{r+p^l}$ have different $I$-eigenvalues, it is enough to show that $\{s_n^{q-1-r+p^l}\}$ and $\{t_n^{r+p^l}\}$, for $n\geq 2$ and $0\leq l\leq f-1$, are linearly independent.    

We show that the vectors in 
\[\{s_n^{q-1-r+p^l}\}_{n\geq 2, 0\leq l \leq f-1}\]
are linearly independent, and the proof for $\{t_n^{r+p^l}\}$ is similar. Suppose that
\[\sum\limits_{\substack{i=2}}^nc_is_i^{q-1-r+p^l}\in(\Ker~T_{-1,0},\Ker~T_{1,2})\]
where $c_i\in\overline{{\mathbb F}}_p$ and $n\in\mathbb{N}$. Since no reduction is possible in the above expression and also these vectors obviously cannot be in $\Ker~T_{-1,0}$, it follows that \[\sum\limits_{\substack{i=2}}^nc_is_i^{q-1-r+p^l}\in\Ker~T_{1,2}.\] For $i\neq j$ with $2\leq i,j\leq n$, once again from the formula for $T_{1,2}$, there cannot be any cancellation between $T_{1,2}(c_is_i^{q-1-r+p^l})$ and $T_{1,2}(c_js_j^{q-1-r+p^l}),$ so we get
\[c_is_i^{q-1-r+p^l}\in\Ker~T_{1,2}\] for all $2\leq i\leq n.$ By Lemma \ref{kernel condition} (1), it follows that $c_i=0$ for all $2\leq i\leq n$. 
\end{proof}

\begin{remark}\label{rmk-li}
It follows by eigenvalue considerations as in the proof of Proposition \ref{linear independence} that the set
\[\mathcal S_2 \cup \{[ {\rm Id},1],[\beta,1]\} \cup \mathcal T_2\]
is linearly independent.
\end{remark}

\subsection{Auxiliary lemmas}

We will have to make use of the following elementary lemma \cite[Lemma 2.8]{hen18}.

\begin{lemma}\label{coeff predicting}
Let $n\geq 1$ and $\phi:I_n\rightarrow \overline{\mathbb{F}}_p$ be any set map. Then there exists a unique polynomial $Q(x_0,\dots,x_{n-1})\in \overline{\mathbb{F}}_p[x_0, x_1, \dots, x_{n-1}]$ in which degree of each variable is at most $q-1$ and $\phi(\mu)=Q(\mu_0, \mu_1, \dots, \mu_{n-1})$ for all 
$\mu\in I_n.$ 
\end{lemma}

The next two lemmas are the first steps towards the proof of Theorem \ref{thm-main}.

\begin{lemma}\label{possible mu power in kk}
Let $\mu = [\mu_0] + [\mu_1]\varpi + \dots + [\mu_{n-1}] \varpi^{n-1} \in I_n$ and $r=r_0+r_1p+\dots+r_{f-1}p^{f-1}$ with $0<r_j<p-1$ for all $0\leq j\leq f-1.$ Let 
\[f_n=f_n'+f_n''\]
be such that  \[f_n'=\sum\limits_{\substack{\mu\in I_n}} a(\mu_0,\mu_1,\dots,\mu_{n-1})
\left[\left(\begin{array}{cc}
\varpi^n& \mu\\
0 & 1
\end{array}\right),1\right]\] and 
\[f_n''=\sum\limits_{\substack{\mu\in I_n}}b(\mu_0,\mu_1,\dots,\mu_{n-1})
\left[\left(\begin{array}{cc}
\varpi^{n-1}& [\mu]_{n-1}\\
0      & 1
\end{array}\right)
\left(\begin{array}{cc}
1& [\mu_{n-1}]\\
0 & 1
\end{array}\right)w,1\right],\] where $a(\mu_0,\dots,\mu_{n-1})$ and $b(\mu_0,\dots,\mu_{n-1})$ are polynomials in $\mu_0,\dots,\mu_{n-1}.$ Suppose 
\[\left(\begin{array}{cc}
1 & -\varpi^{n-1}\\
0 &  1
\end{array}\right)f_n-f_n\in (\Ker~T_{-1,0},\Ker~T_{1,2}).\] 
Then
\begin{enumerate}
\item[{\rm(1)}] the possible powers of $\mu_{n-1},$ say $k = k_0+k_1p+\dots+k_{f-1}p^{f-1}$, in $a(\mu_0,\dots,\mu_{n-1})$ will satisfy one of the following three conditions:
\begin{enumerate}
\item[{\rm (a)}] there exists some $0\leq j'\leq f-1$ such that $k_{j'}<p-1-r_{j'},$ 
\item[{\rm (b)}] $k_j=p-1-r_j$ for all $0\leq j\leq f-1,$
\item[{\rm (c)}] $k_j=p-1-r_j$ for $j\neq l$ and $k_l=p-r_l$ for some $0\leq l\leq f-1.$
\end{enumerate}
\item[{\rm(2)}] the possible powers of $\mu_{n-1},$ say $k = k_0+k_1p+\dots+k_{f-1}p^{f-1}$, in $b(\mu_0,\dots,\mu_{n-1})$ will satisfy one of the following three conditions:
\begin{enumerate}
\item[{\rm (a)}] there exists some $0\leq j'\leq f-1$ such that such that $k_{j'}<r_{j'},$ 
\item[{\rm (b)}] $k_j=r_j$ for all $0\leq j\leq f-1,$
\item[{\rm (c)}] $k_j=r_j$ for $j\neq l$ and $k_l=r_l+1$ for some $0\leq l\leq f-1.$
\end{enumerate}
\end{enumerate}
\end{lemma}

\begin{proof}[Proof of Lemma \ref{possible mu power in kk}]
We will prove $(1)$ and the proof of $(2)$ is similar. Suppose $(1)$ does not hold. Then there exists $k$ such that $k_j\geq p-1-r_j$ for all $0\leq j\leq f-1$ with 
\[k_{j_0}>p-1-r_{j_0} \mbox{~for~some~} 0 \leq j_0 \leq f-1 ~\&~ k \neq (p-r_{j_0})p^{j_0} +\sum\limits_{\substack{j_0\neq j=0}}^{f-1}(p-1-r_j)p^j.\]
Then either there exists $j_1$ with $j_1\neq j_0$ such that $k_{j_1}>p-1-r_{j_1}$ or 
\[k= k_{j_0}p^{j_0}+\sum\limits_{\substack{j_0\neq j=0}}^{f-1}(p-1-r_j)p^j\]
with $k_{j_0}>p-r_{j_0}.$ Choose $k$ with the above property such that there is no other monomial $\mu_{n-1}^{k'}$ in $a(\mu_0,\dots,\mu_{n-1})$ with $k_j\leq k'_j$ for all $0\leq j\leq f-1$. Since a polynomial is of finite degree, such a $k$ exists. Let 
\[g=\left(\begin{array}{cc}
1 & -\varpi^{n-1}\\
0 &   1
\end{array}\right).\]
We have
\[gf_n-f_n=(gf_n'-f_n')+(gf_n''-f_n'')\in(\Ker~T_{-1,0},\Ker~T_{1,2}).\]
Note that,
\begin{align*}
gf_n^\prime -f_n^\prime &=\sum\limits_{\substack{\mu\in I_n}}\left[a([\mu]_{n-1},\mu_{n-1}+1)-a([\mu]_{n-1},\mu_{n-1})\right]\left[\left(\begin{array}{cc}
\varpi^n & \mu\\
0  & 1
\end{array}\right),1\right]
\end{align*}      
and 
\begin{align*}
& gf_n''-f_n'' \\
&=\sum\limits_{\substack{\mu\in I_n}}\left[b([\mu]_{n-1},\mu_{n-1}+1)-b([\mu]_{n-1},\mu_{n-1})\right] \left[\left(\begin{array}{cc}
\varpi^{n-1}& [\mu]_{n-1}\\
0      & 1
\end{array}\right)
\left(\begin{array}{cc}
1& [\mu_{n-1}]\\
0 & 1
\end{array}\right)w,1\right].
\end{align*}

Let \[\Delta a=a([\mu]_{n-1},\mu_{n-1}+1)-a([\mu]_{n-1},\mu_{n-1})\]
considered as a polynomial in $\mu_{n-1}$ with coefficients in $\overline{\mathbb F}_p[\mu_0,\dots,\mu_{n-2}]$. By Theorem \ref{lucas}, we have 
\[(\mu_{n-1}+1)^k-\mu_{n-1}^k \equiv \sum\limits_{\substack{i=0}}^{k-1}\prod_{j=0}^{f-1}{k_j\choose i_j}\mu_{n-1}^i \mod p.\] 

Now if there exists $j_1$ with $j_1\neq j_0$ such that $k_{j_1}>p-1-r_{j_1},$ take 
\[k'= (k_{j_1}-1)p^{j_1}+\sum\limits_{\substack{j_1\neq j=0}}^{f-1}k_jp^j.\]
The coefficient of $\mu_{n-1}^{k'}$ in $\Delta a$ is 
\[{k\choose k'}={k_{j_1}\choose k_{j_1}-1}\not\equiv 0 \mod p\]
by Theorem \ref{lucas} and Corollary \ref{prime divides bio coeff}. Note that the term involving $\mu_{n-1}^{k'}$ in $gf_n'-f_n'$ cannot get cancelled by any other term in $gf_n-f_n$. Indeed, it cannot get cancelled with any other term in $gf_n^\prime-f_n^\prime$ because of the choice of $k$ and anyway no term in $gf_n^\prime-f_n^\prime$ can get cancelled with a term in $gf_n^{\prime\prime}-f_n^{\prime\prime}$ (pictorially they represent edges of opposite orientation on the tree of ${\rm SL}_2(F)$). So this term involving $\mu_{n-1}^{k'}$ must be there in $(\Ker~T_{-1,0},\Ker~T_{1,2}),$ but then Lemma \ref{kernel condition} (1) would imply that there exists some $0\leq l\leq f-1$ such that $k'_l<p-1-r_l,$ which contradicts our assumption. So $k$ must be of the form \[k= k_{j_0}p^{j_0}+\sum\limits_{\substack{j_0\neq j=0}}^{f-1}(p-1-r_j)p^j\]
with $k_{j_0}>p-r_{j_0}.$ Taking 
\[k'= (k_{j_0}-1)p^{j_0}+\sum\limits_{\substack{j_0\neq j=0}}^{f-1}(p-1-r_j)p^j,\]
and using the same argument as in the previous case, we arrive at a contradiction.
\end{proof}

\begin{remark}\label{on hendel's paper}
The idea of choosing $k$ as in Lemma \ref{possible mu power in kk} is already employed by Hendel in \cite[Lemma 3.13]{hen18}.
\end{remark}

Now we state one more lemma whose main idea of proof also comes from \cite[Lemma 3.13]{hen18}. In what follows, $B(t)$ denotes the ball of radius $m$ on the tree of ${\rm SL}_2(F)$ with center at the vertex representing the trivial coset $G/KZ$. Explicitly it consists of linear combinations of vectors of the form  
\[B^0(t) = \left\{[g_{n,\mu}^0,1], \left[g_{n-1,[\mu]_{n-1}}^0 
\left(\begin{array}{cc}
1& [\mu_{n-1}] \\
0 & 1
\end{array}\right)w,1\right]\right\}_{n \leq t},\]
and
\[B^1(t) =\left\{[g_{n-1,[\mu]_{n-1}}^1w,1], \left[g_{n-2,[\mu]_{n-2}}^1w 
\left(\begin{array}{cc}
1& [\mu_{n-2}] \\
0 & 1
\end{array}\right)w,1\right]\right\}_{n \leq t},\]
where $\mu = [\mu_0]+[\mu_1]\varpi+\dots+[\mu_{n-1}]\varpi^{n-1}  \in I_n$.

\begin{lemma}\label{reduce to one variable for kk}
Let \[f_n'=\sum\limits_{\substack{\mu\in I_n}}\sum \limits_{\substack{l=0}}^{f-1}P_l([\mu]_{n-1})\mu_{n-1}^{q-1-r+p^l}\left[g^0_{n,\mu},1\right]\] and \[f_n''=\sum\limits_{\substack{\mu\in I_n}}\sum\limits_{\substack{l=0}}^{f-1}Q_l([\mu]_{n-1})\mu_{n-1}^{r+p^l}\left[g^0_{n-1,[\mu]_{n-1}}\left(\begin{array}{cc}
1& [\mu_{n-1}]\\
0 & 1
\end{array}\right)w,1\right],\] where $P_l([\mu]_{n-1})$ and $Q_l([\mu]_{n-1})$ are polynomials in $\mu_0,\dots,\mu_{n-2}$. Let $f_n=f_n'+f_n''$. Let $f=f_n+f'$ be such that $f'\in B(n-1)$ and 
\[\left(\begin{array}{cc}
1 & -\varpi^{n-m}\\
0 &   1
\end{array}\right)f-f\in (\Ker~T_{-1,0},\Ker~T_{1,2}),\] for all $1\leq m\leq n-1.$
Then we have
\[f_n'=\sum\limits_{\substack{\mu\in I_n}}\sum \limits_{\substack{l=0}}^{f-1}a_l\mu_{n-1}^{q-1-r+p^l}\left[g^0_{n,\mu},1\right]\] and \[f_n''=\sum\limits_{\substack{\mu\in I_n}}\sum\limits_{\substack{l=0}}^{f-1}b_l\mu_{n-1}^{r+p^l}\left[g^0_{n-1,[\mu]_{n-1}}\left(\begin{array}{cc}
1& [\mu_{n-1}]\\
0 & 1
\end{array}\right)w,1\right],\] where $a_l$ and $b_l$ are constants. 
\end{lemma}

\begin{proof}[Proof of Lemma \ref{reduce to one variable for kk}]
We do the proof only for $f_n^\prime$, as the case of $f_n''$ is similar. 
The proof is by induction on $n$. Note that $P_l([\mu]_{n-1})$ is independent of $\mu_{n-1}.$ Suppose it is independent of $\mu_{n-1},\dots,\mu_{n-m+1}.$ Then
\[f_n'=\sum\limits_{\substack{\mu\in I_n}}\sum \limits_{\substack{l=0}}^{f-1}P_l([\mu]_{n-m},\mu_{n-m})\mu_{n-1}^{q-1-r+p^l}\left[g^0_{n,\mu},1\right]\]
We show that it is independent of $\mu_{n-m}$. It is given to us that 
\begin{align*}
\left(\begin{array}{cc}
1 & -\varpi^{n-m}\\
0 &   1
\end{array}\right)f-f &=\left[\left(\begin{array}{cc}
1 & -\varpi^{n-m}\\
0 &   1
\end{array}\right)f_n-f_n\right]+\left[\left(\begin{array}{cc}
1 & -\varpi^{n-m}\\
0 &   1
\end{array}\right)f'-f'\right]\\
& \in  (\Ker~T_{-1,0},\Ker~T_{1,2}).
\end{align*}
Now,
\begin{align*}
\left(\begin{array}{cc}
1 & -\varpi^{n-m}\\
0 & 1
\end{array}\right)
\left(\begin{array}{cc}
\varpi^n & \displaystyle{\sum_{i=0}^{n-1}} [\mu_i]\varpi^i \\
0   & 1
\end{array}\right) &=\left(\begin{array}{cc}
\varpi^n & [\mu_0]+\dots+[\mu_{n-1}]\varpi^{n-1}-\varpi^{n-m}\\
0   & 1
\end{array}\right)
\end{align*}
and this equals
\begin{align*}
\left(\begin{array}{cc}
\varpi^n & \displaystyle{\sum_{i=0}^{n-m-1}} [\mu_i]\varpi^i+[\mu_{n-m}-1]\varpi^{n-m}+[\mu'_{n-m+1}]\varpi^{n-m+1}+\dots+[\mu'_{n-1}]\varpi^{n-1}\\
0   & 1
\end{array}\right)\\
\end{align*} 
where $\mu'_k=\mu_k+c_k(\mu_{n-m},\dots,\mu_{n-2})$ for $n-m+1\leq k\leq n-1.$ 
	
Note that the transformation $\mu'_k\mapsto \mu_k-c_k(\mu_{n-m},\dots,\mu_{n-2})$ does not affect the variables $\mu_k$ for $n-m+1\leq k\leq n-1$ in $P_l([\mu]_{n-1}),$ as it is independent of these variables. This transformation together with $\mu_{n-m}\mapsto \mu_{n-m}+1$ gives 
\[\left(\begin{array}{cc}
1 & -\varpi^{n-m}\\
0 & 1
\end{array}\right)f_n'=
\sum\limits_{\substack{\mu\in I_n}}\sum \limits_{\substack{l=0}}^{f-1}P_l([\mu]_{n-m},\mu_{n-m}+1)(\mu_{n-1}-c_{n-1})^{q-1-r+p^l} [g_{n,\mu}^0,1].\] 
In the above expression, by $c_{n-1}$ we mean $c_{n-1}(\mu_{n-m},\dots,\mu_{n-2})$. 
Now,
\begin{align*}
\left(\begin{array}{cc}
1 & -\varpi^{n-m}\\
0 & 1
\end{array}\right)f_n'-f_n' & =\sum\limits_{\substack{\mu\in I_n}}\sum \limits_{\substack{l=0}}^{f-1} \alpha(\mu,l)[g^0_{n,\mu},1],
\end{align*}
where 
\[\alpha(\mu,l)=\left[P_l([\mu]_{n-m},\mu_{n-m}+1)(\mu_{n-1}-c_{n-1})^{q-1-r+p^l}-P_l([\mu]_{n-m},\mu_{n-m})\mu_{n-1}^{q-1-r+p^l}\right].\]
Thus,
\begin{align*}
& \left(\begin{array}{cc}
1 & -\varpi^{n-m}\\
0 & 1
\end{array}\right)f_n'-f_n' \\
& = \sum\limits_{\substack{\mu\in I_n}}\sum \limits_{\substack{l=0}}^{f-1}\left[P_l([\mu]_{n-m},\mu_{n-m}+1)-P_l([\mu]_{n-m},\mu_{n-m})\right]\mu_{n-1}^{q-1-r+p^l}[g^0_{n,\mu},1]\\
&+\sum\limits_{\substack{\mu\in I_n}}\sum \limits_{\substack{l=0}}^{f-1}\sum\limits_{\substack{i=0}}^{q-1-r+p^l-1} \beta(\mu,l,i) [g^0_{n,\mu},1],
\end{align*}
where
\[\beta(\mu,l,i)= P_l([\mu]_{n-m},\mu_{n-m}+1)(-1)^i{q-1-r+p^l\choose i}(-c_{n-1})^{q-1-r+p^l-i}\mu_{n-1}^i.\]
Now we read this modulo $\Ker~ T_{1,2}$. Thus, we get
\begin{align*}
& \left(\begin{array}{cc}
1 & -\varpi^{n-m}\\
0 & 1
\end{array}\right)f_n'-f_n' \\
& = \sum\limits_{\substack{\mu\in I_n}}\sum \limits_{\substack{l=0}}^{f-1}\left[P_l([\mu]_{n-m},\mu_{n-m}+1)-P_l([\mu]_{n-m},\mu_{n-m})\right]\mu_{n-1}^{q-1-r+p^l}[g^0_{n,\mu},1]\\
&+\sum\limits_{\substack{\mu\in I_n}}\sum \limits_{\substack{l=0}}^{f-1}P_l([\mu]_{n-m},\mu_{n-m}+1){q-1-r+p^l\choose q-1-r}(-c_{n-1})^{p^l}\mu_{n-1}^{q-1-r}[g^0_{n,\mu},1],
\end{align*}
by Corollary \ref{prime divides bio coeff} and Lemma \ref{kernel condition} (1), exactly as we have argued before in the proof of Lemma \ref{ball two invariants}. Now by Lemma \ref{modulo kk} (1), it follows that, modulo $(\Ker~T_{-1,0},\Ker~T_{1,2})$, we have
\begin{align*}
& \left(\begin{array}{cc}
1 & -\varpi^{n-m}\\
0 & 1
\end{array}\right)f_n'-f_n' \\
& = \sum\limits_{\substack{\mu\in I_n}}\sum \limits_{\substack{l=0}}^{f-1}\left[P_l([\mu]_{n-m},\mu_{n-m}+1)-P_l([\mu]_{n-m},\mu_{n-m})\right]\mu_{n-1}^{q-1-r+p^l}[g^0_{n,\mu},1]+g_{n-1}
\end{align*}
where $g_{n-1} \in B(n-1)$. As $r_l\neq 0,$ by Lemmas \ref{kernel condition} (1) and \ref{modulo kk} (1) we have 
\[\sum\limits_{\substack{\mu\in I_n}}\mu_{n-1}^{q-1-r+p^l}\left[g^0_{n,\mu},1\right]\notin (\Ker~T_{-1,0},\Ker~T_{1,2}).\] Also the term involving $\mu_{n-1}^{q-1-r+p^l}$ cannot get cancelled by any other term in the expression \[\left(\begin{array}{cc}
1 & -\varpi^{n-m}\\
0 &   1
\end{array}\right)f-f.\] So it follows that \[P_l([\mu]_{n-m},\mu_{n-m}+1)-P_l([\mu]_{n-m},\mu_{n-m})=0.\] Hence $P_l([\mu]_{n-1})$ is independent of $\mu_{n-m}.$ Therefore, by induction $P_l([\mu]_{n-1})$ is a constant. 
\end{proof}

\subsection{Proof of Theorem \ref{thm-main}}

Clearly the vectors $\left[{\rm Id}, 1\right]$ and $\left[\beta, 1\right]$ are fixed by $I(1).$ By Lemmas \ref{enough to look at ball one for kk} and \ref{ball two invariants} and Remark \ref{same computation} the vectors in $\mathcal{S}_2$ and $\mathcal T_2$ are $I(1)$-invariant modulo $(\Ker~T_{-1,0},\Ker~T_{1,2})$ except for the case when both $e=1$ and $f=1.$ By Remark \ref{rmk-li}, the set
$\mathcal{S}_2 \cup \{[{\rm Id},1],[\beta,1]\} \cup \mathcal T_2$ is linearly independent. 

Now let $f\in {\rm ind}_{IZ}^G \chi_r$ be an $I(1)$-invariant of
\[\pi_r = \frac{{\rm ind}_{IZ}^G\chi_r}{(\Ker~T_{-1,0},\Ker~T_{1,2})}. \]
We write 
\[f=f^0+f^1\]
where $f^0$ (resp. $f^1$) is a linear combination of vectors on the zero side (resp. one side) of the tree of ${\rm SL}_2(F)$. By this, we mean $f^0$ is a linear combination of vectors of the form
\[[ g_{n,\mu}^0,1], \left[g_{n-1,[\mu]_{n-1}}^0 
\left(\begin{array}{cc}
1& [\mu_{n-1}] \\
0 & 1
\end{array}\right)w,1\right]\]
and $f^1$ is a linear combination of vectors of the form
\[ [g_{n-1,[\mu]_{n-1}}^1w,1], \left[g_{n-2,[\mu]_{n-2}}^1w 
\left(\begin{array}{cc}
1& [\mu_{n-2}] \\
0 & 1
\end{array}\right)w,1\right].\]

Then, \[gf^i-f^i\in (\Ker~T_{-1,0}, \Ker~T_{1,2}),\] for all $i\in \{0,1\}$ and $g\in I(1)$. Since $\beta f^1$ is a linear combination of vectors on the zero side and $\beta$ normalizes $I(1)$, without loss of generality, we may assume $f=f^0$. Write
\[f=f_n+f'\]
with $f_n\neq 0$, $f'\in B(n-1)$, for $n$ maximal. Now, 
\[f_n=\sum\limits_{\substack{\mu\in I_n}}a_{{\mu}}\left[g^0_{n,\mu},1\right]+\sum\limits_{\substack{\mu\in I_n}}b_{{\mu}}\left[g^0_{n-1,[\mu]_{n-1}}\left(\begin{array}{cc}
1& [\mu_{n-1}]\\
0 & 1
\end{array}\right)w,1\right],\]
where $\mu=[\mu_0]+[\mu_1]\varpi+\dots+[\mu_{n-1}]\varpi^{n-1}$ and $a_{{\mu}}, b_{{\mu}}\in \overline{\mathbb F}_p.$ By Lemma \ref{coeff predicting}, the coefficients $a_{{\mu}}$ and $b_{{\mu}}$ can be replaced by the polynomials $a(\mu_0,\dots,\mu_{n-1})$ and  $b(\mu_0,\dots,\mu_{n-1})$ respectively, where each $\mu_i$ has maximum degree $q-1.$ Write 
\[f_n=f_n'+f_n'',\]
where 
\[f_n'=\sum\limits_{\substack{\mu\in I_n}}\sum\limits_{\substack{i}}
a(i_0,i_1,\dots,i_{n-1})\mu_0^{i_0}\dots\mu_{n-1}^{i_{n-1}}
\left[g^0_{n,\mu},1\right],\] 
and 
\[f_n''=\sum\limits_{\substack{\mu\in I_n}}\sum\limits_{\substack{j}}
b(j_0,j_1,\dots,j_{n-1})\mu_0^{j_0}\dots\mu_{n-1}^{j_{n-1}}
\left[g^0_{n-1,[\mu]_{n-1}}
\left(\begin{array}{cc}
1& [\mu_{n-1}]\\
0 & 1
\end{array}\right)w,1\right].\]

Let
\[g'=\left(\begin{array}{cc}
1 & -\varpi^{n-1}\\
0 & 1
\end{array}\right)\in I(1).\]
Since $f'$ belongs in $B(n-1)$, it is easy to check that $g'$ fixes $f'$. This gives
\[g'f_n-f_n\in (\Ker~T_{-1,0},\Ker~T_{1,2}).\] Now Lemma \ref{kernel condition} (1) together with Lemma \ref{possible mu power in kk} (1) gives 
\begin{align*}
f_n' &= \sum\limits_{\substack{\mu\in I_n}}\sum\limits_{\substack{i}}a(i_0,\dots,i_{n-2},q-1-r)\mu_0^{i_0}\dots\mu_{n-1}^{q-1-r}\left[g^0_{n,\mu},1\right] \\
&+ \sum\limits_{\substack{\mu\in I_n}}\sum\limits_{\substack{l=0}}^{f-1}a_l([\mu]_{n-1})\mu_{n-1}^{q-1-r+p^l}\left[g^0_{n,\mu},1\right],
\end{align*}
which in turn implies that 
\begin{align*}
& f_n'-\sum\limits_{\substack{\mu\in I_n}}\sum\limits_{\substack{l=0}}^{f-1}a_l([\mu]_{n-1})\mu_{n-1}^{q-1-r+p^l}\left[g^0_{n,\mu},1\right]\\
&=\sum\limits_{\substack{\mu_0,\dots,\mu_{n-1}}}\sum\limits_{\substack{i_0,\dots,i_{n-2}}}a(i_0,\dots,q-1-r)\mu_0^{i_0}\dots\mu_{n-1}^{q-1-r}\left[g^0_{n,\mu},1\right]
\end{align*}
which modulo $\Ker~ T_{1,2}$ equals
\begin{align*}
\sum\limits_{\substack{\mu_0,\dots,\mu_{n-2}}}\sum\limits_{\substack{i_0,\dots,i_{n-2}}}a(i_0,\dots,q-1-r)\mu_0^{i_0}\dots\mu_{n-2}^{i_{n-2}} 
\left[g^0_{n-2,[\mu]_{n-2}}\left(\begin{array}{cc}
1& [\mu_{n-2}]\\
0& 1
\end{array}\right)w,1\right],                        
\end{align*}  
by Lemma \ref{modulo kk} (1). This vector belongs to $B(n-1)$ which we call $g_{n-1}^\prime$. We get 
\[f_n'=\sum\limits_{\substack{\mu\in I_n}}\sum\limits_{\substack{l=0}}^{f-1}a_l([\mu]_{n-1})\mu_{n-1}^{q-1-r+p^l}\left[g^0_{n,\mu},1\right]+g_{n-1}'.\] 
Similarly, working with $f_n'',$ we get
\[f_n''=\sum\limits_{\substack{\mu\in I_n}}\sum\limits_{\substack{l=0}}^{f-1}b_l([\mu]_{n-1})\mu_{n-1}^{r+p^l}\left[g_{n-1,[\mu]_{n-1}}^0 
\left(\begin{array}{cc}
1 & [\mu_{n-1}] \\
0 & 1
\end{array}\right)w,1\right]+g_{n-1}''\] for some $g_{n-1}''\in B(n-1)$, by Lemmas \ref{kernel condition} (2), \ref{possible mu power in kk} (2) and \ref{modulo kk} (2). 

For $1\leq m\leq n-1,$ we note that 
\[\left(\begin{array}{cc}
1 & -\varpi^{n-m}\\
0 &   1
\end{array}\right)\in I(1).\]
Using the condition 
\[\left(\begin{array}{cc}
1 & -\varpi^{n-m}\\
0 &   1
\end{array}\right)f-f\in (\Ker~T_{-1,0},\Ker~T_{1,2}),\] 
by Lemma \ref{reduce to one variable for kk}, we have 
\[f_n'=\sum\limits_{\substack{\mu\in I_n}}\sum\limits_{\substack{l=0}}^{f-1} a_{l,n}\mu_{n-1}^{q-1-r+p^l}\left[g^0_{n,\mu},1\right]+g_{n-1}'\]
and
\[f_n''=\sum\limits_{\substack{\mu\in I_n}}\sum\limits_{\substack{l=0}}^{f-1}b_{l,n}\mu_{n-1}^{r+p^l}\left[g_{n-1,[\mu]_{n-1}}^0 
\left(\begin{array}{cc}
1 & [\mu_{n-1}] \\
0 & 1
\end{array}\right)w,1\right]+g_{n-1}'',\] 
where $a_l$ and $b_l$ are constants. 

Hence $f_n$ takes the form
\[f_n=\sum\limits_{\substack{l=0}}^{f-1}a_{l,n}s_n^{q-1-r+p^l}+\sum\limits_{\substack{l=0}}^{f-1}b_{l,n}t
_n^{r+p^l}+g_{n-1},\]
where 
\[g_{n-1}=g_{n-1}'+g_{n-1}''\in B(n-1).\]
Thus it follows that \[f-\sum\limits_{\substack{l=0}}^{f-1}a_{l,n}s_n^{q-1-r+p^l}-\sum\limits_{\substack{l=0}}^{f-1}b_{l,n}t_n^{r+p^l}=g_{n-1}+f'\] is an $I(1)$-invariant vector modulo $(\Ker~T_{-1,0},\Ker~T_{1,2})$ in $B(n-1)$. 

Applying this argument on vectors in $B(n-1)$ and repeating this process, we get                
\begin{align*}
&f=\sum\limits_{\substack{l=0}}^{f-1}a_{l,n} s_n^{q-1-r+p^l}
+\sum\limits_{\substack{l=0}}^{f-1}b_{l,n}t_n^{r+p^l}
+\dots+\sum\limits_{\substack{l=0}}^{f-1}a_{l,2}s_2^{q-1-r+p^l}
+\sum\limits_{\substack{l=0}}^{f-1}b_{l,2}t_2^{r+p^l}
+f_1,
\end{align*}
where $f_1$ is an $I(1)$-invariant in $B(1)$. 
Write 
\[f_1=f_1'+f_1'',\]
where \[f_1'=\sum\limits_{\substack{\mu\in I_1}}\sum\limits_{\substack{i}}
a_i\mu^i\left[g^0_{1,\mu},1\right],\] 
and 
\[f_1''=\sum\limits_{\substack{\mu\in I_1}}\sum\limits_{\substack{j}}
b_j\mu^j
\left[\left(\begin{array}{cc}
1      & \mu\\
0      & 1
\end{array}\right)
w,1\right].\]
Using the action of 
\[u= \left(\begin{array}{cc}
1 & 1\\
0 & 1
\end{array}\right)\]
on $f_1$, by Lemma \ref{possible mu power in kk} (1), the possible powers $i$ of $\mu$ in $f_1'$ will satisfy either $0\leq i\leq q-1-r$ or $i=q-1-r+p^l$ for some $0\leq l\leq f-1.$ If $i=q-1-r+p^l,$ then \[\left(\begin{array}{cc}
1 & 1\\
0 & 1
\end{array}\right)f_1'-f_1'=\sum\limits_{\substack{\mu\in I_1}}\mu^{q-1-r}\left[g^0_{1,\mu},1\right]\in 
(\Ker~T_{-1,0},\Ker~T_{1,2}).\] 
This, by Lemma \ref{modulo kk} (1), gives $\left[\beta,1\right]\in (\Ker~T_{-1,0},\Ker~T_{1,2}),$ which is not possible. So we must have $0\leq i\leq q-1-r.$ Then, by Lemma \ref{kernel condition}(1) and Lemma \ref{modulo kk} (1), we have \[f_1'=\left[\beta,1\right]\mod (\Ker~T_{-1,0},\Ker~T_{1,2}).\] Similarly, by Lemmas \ref{kernel condition} (2) and \ref{possible mu power in kk} (2) and Lemma \ref{modulo kk} (2) , we can show that
\[f_1''=\left[{\rm Id}, 1\right]\mod (\Ker~T_{-1,0},\Ker~T_{1,2}).\]
Thus, we have
\begin{align*}
&f=\sum\limits_{\substack{l=0}}^{f-1}a_{l,n} s_n^{q-1-r+p^l}
+\sum\limits_{\substack{l=0}}^{f-1}b_{l,n}t_n^{r+p^l}+\dots \\
&+\sum\limits_{\substack{l=0}}^{f-1}a_{l,2}s_2^{q-1-r+p^l}
+\sum\limits_{\substack{l=0}}^{f-1}b_{l,2}t_2^{r+p^l}+ c\left[\beta,1\right]+d\left[{\rm I}d,1\right].
\end{align*}

Now assume $e=1$ and $f=1$. Let $f\in {\rm ind}_{IZ}^G\chi_r$ be an $I(1)$-invariant vector modulo $(\Ker~ T_{-1,0}, \Ker~ T_{1,2}).$ As in the previous case, we concentrate only on the zero side of the tree and assume that $f=f^0.$ We write $f=f_n+f'$ where $f_n\neq 0$ and $f'\in B(n-1).$ We further write $f_n=f_n'+f_n''$ where $f_n'$ and $f_n''$ are same as in the previous case. Following the steps in the previous case, we have 
\[f_n'=\sum\limits_{\substack{\mu\in I_n}} a_0\mu_{n-1}^{p-r}\left[g^0_{n,\mu},1\right]+g_{n-1}',\]
and
\[f_n''=\sum\limits_{\substack{\mu\in I_n}}b_0\mu_{n-1}^{r+1}\left[g_{n-1,[\mu]_{n-1}}^0 
\left(\begin{array}{cc}
1 & [\mu_{n-1}] \\
0 & 1
\end{array}\right)w,1\right]+g_{n-1}'',\] where $a_0$ and $b_0$ are constants and $g_{n-1}', g_{n-1}''\in B(n-1).$ Thus, 
\[f_n=\sum\limits_{\substack{\mu\in I_n}} a_0\mu_{n-1}^{p-r}\left[g^0_{n,\mu},1\right]+\sum\limits_{\substack{\mu\in I_n}}b_0\mu_{n-1}^{r+1}\left[g_{n-1,[\mu]_{n-1}}^0 
\left(\begin{array}{cc}
1 & [\mu_{n-1}] \\
0 & 1
\end{array}\right)w,1\right]+g_{n-1},\]
where $g_{n-1}=g_{n-1}'+g_{n-1}''\in B(n-1).$ Write $f=f_n+f_{n-1}+f'.$ We get
\begin{align*}
\left(\begin{array}{cc}
1 & p^{n-2}\\
0 & 1
\end{array}\right)f-f &= \left[\left(\begin{array}{cc}
1 & p^{n-2}\\
0 & 1
\end{array}\right)f_n-f_n\right]+
\left[\left(\begin{array}{cc}
1 & p^{n-2}\\
0 & 1
\end{array}\right)f_{n-1}-f_{n-1}\right] \\
&\in (\Ker~ T_{-1,0}, \Ker~ T_{1,2}).
\end{align*}
For $e=1$, we have
\begin{align*}
\left(\begin{array}{cc}
1 & p^{n-2}\\
0 & 1
\end{array}\right)f_n'-f_n' &=\sum\limits_{\substack{\mu\in I_n}} a_0\left[(\mu_{n-1}-(*))^{p-r}-\mu_{n-1}^{p-r}\right]\left[g^0_{n,\mu},1\right],
\end{align*}
where 
\[(*)=\sum\limits_{\substack{s=1}}^{p-1}(-1)^{p-s}\frac{{p\choose s}}{p}\mu_{n-2}^s.\]
Then, by Lemmas \ref{modulo kk} (1) and \ref{kernel condition} (1), modulo 
$(\Ker~T_{-1,0}, \Ker~T_{1,2})$, the above expression becomes
\begin{equation}\label{equ1}
-\sum\limits_{\substack{\mu\in I_{n-1}}}a_0{p-r\choose p-1-r}(*)
\left[g^0_{n-2,[\mu]_{n-2}}\left(\begin{array}{cc}
1& [\mu_{n-2}]\\
0& 1
\end{array}\right)w,1\right].
\end{equation}
Writing $f_{n-1}=f_{n-1}'+f_{n-1}''$, we have,
\[\left(\begin{array}{cc}
1 & p^{n-2}\\
0 & 1
\end{array}\right)f_{n-1}-f_{n-1}=
\left[\left(\begin{array}{cc}
1 & p^{n-2}\\
0 & 1
\end{array}\right)f_{n-1}'-f_{n-1}'\right]+\left[\left(\begin{array}{cc}
1 & p^{n-2}\\
0 & 1
\end{array}\right)f_{n-1}''-f_{n-1}''\right].\]
No term in the first summand of the above equation can cancel a term in (\ref{equ1}). Also, by Lemma \ref{possible mu power in kk} (2), the possible powers, say $k,$ of $\mu_{n-2}$ in $f_{n-1}''$ must satisfy either $0\leq k\leq r$ or $k=r+1.$ As $r<p-1$, we have $\max (r+1)=p-1.$ So the maximum power of $\mu_{n-2}$ in the second summand of the above equation is $p-2.$ In both the cases, the term involving $\mu_{n-2}^{p-1}$ in (\ref{equ1}) will not get cancelled. Since there is no reduction, this term must be in $\Ker~T_{-1,0},$ which is not possible by Lemma \ref{kernel condition} (2). Thus we arrive at a contradiction. So $i_{n-1}$ can not be $p-r.$ Thus one can always modify $f_n'$ by a vector $g_{n-1}'$ in $B(n-1).$ Similarly, working with $f_n'',$ we can modify it by a vector $g_{n-1}''$ in $B(n-1).$ Thus $f_n$ is congruent to a vector $f_{n-1}$ in $B(n-1)$ modulo $(\Ker~T_{-1,0}, \Ker~T_{1,2})$ and hence by induction, $f$ is congruent to a vector $f_1$ in $B(1)$ modulo $(\Ker~T_{-1,0}, \Ker~T_{1,2}).$ Write $f_1=f_1'+f_1'',$ where 
\[f_1'=\sum\limits_{\substack{i}}\sum\limits_{\substack{\mu\in I_1}}
a_i\mu^i\left[g^0_{1,\mu},1\right],\]
and 
\[f_1''=\sum\limits_{\substack{j}}\sum\limits_{\substack{\mu\in I_1}}
b_j\mu^j
\left[\left(\begin{array}{cc}
1      & \mu\\
0      & 1
\end{array}\right)
w,1\right].\]
Considering the action of $\left(\begin{array}{cc}
1 & 1\\
0 & 1
\end{array}\right)$ on $f_1$ as in the previous case, we have $0\leq i\leq p-1-r$ and $0\leq j\leq r$, by Lemma \ref{modulo kk} and Lemma \ref{possible mu power in kk}. Then, by Lemma \ref{modulo kk} and Lemma \ref{kernel condition}, modulo $(\Ker~T_{-1,0}, \Ker~T_{1,2}),$ we get $f_1'=\left[\beta, 1\right]$ and $f_1''=\left[{\rm Id}, 1\right].$ Thus we can conclude that
\[f=c\left[{\rm Id}, 1\right]+d\left[\beta, 1\right].\]

This finishes the proof of Theorem \ref{thm-main}.

\subsection{A remark on $\pi_r$}\label{isodetails}

We show that there is no isomorphism between 
\[\tau_r = \frac{{\rm ind}_{KZ}^G \sigma_r}{(T)}\]
and
\[\pi_r=\frac{{\rm ind}_{IZ}^G \chi_r}{({\rm Ker}~ T_{-1,0}, {\rm Ker}~ T_{1,2})}\]
when $f \neq 1$; i.e., $F$ is not a totally ramified extension of $\mathbb Q_p$ (cf. Remark \ref{new}).

Note that any $G$-linear isomorphism 
\[\varphi:\pi_r \rightarrow \tau_r\]
must preserve $I(1)$-invariants and the corresponding $I$-eigenvalues.

Suppose $e=1, f\neq 1$; i.e., $F/\mathbb Q_p$ is unramified. In this case, $s_n^{q-1-r+p^l}$, for $n\geq 2$, is an $I(1)$-invariant in $\pi_r$ such that
\[\left(\begin{array}{cc}
a & b \\ \varpi c & d 
\end{array}\right) \cdot s_n^{q-1-r+p^l} = a^{r-p^l}d^{p^l} \cdot s_n^{q-1-r+p^l}
\]
by Lemma \ref{i action}. By \cite[Theorem 1.2]{hen18}, a basis of the $I(1)$-invariants in $\tau_r$ consists of the vectors 
\[{\rm Id}\otimes\bigotimes_{j=0}^{f-1}x_j^{r_j}, \alpha\otimes\bigotimes_{j=0}^{f-1}y_j^{r_j}, c_n^{p^l(r_l+1)}, \beta  c_n^{p^l(r_l+1)}\]
for $n \geq 1$, where
\[c_n^k = \sum\limits_{\substack{\mu\in I_n}}\left(\begin{array}{cc}
\varpi^n & \mu\\
0        & 1
\end{array}\right)\otimes\mu_{n-1}^k \bigotimes_{j=0}^{f-1}x_j^{r_j}.\]
By \cite[Lemma 3.6]{hen18},
\[\left(\begin{array}{cc}
a & b \\ \varpi c & d 
\end{array}\right) \cdot c_n^k = a^{r-2k}(ad)^k \cdot c_n^k,
\]
and it follows that there is no $I(1)$-invariant vector in $\tau_r$ with $I$-eigenvalue $a^{r-p^l}d^{p^l}$. Thus there is no vector in $\tau_r$ where $s_n^{q-1-r+p^l}$ can be mapped under $\varphi$. This gives a contradiction.

Now, suppose $e>1, f>1$. In this case $t_n^{r+p^l}$, $n\geq 2$, is an $I(1)$-invariant vector in $\pi_r$ with $I$-eigenvalue $a^{q-1-p^l}d^{r+p^l}$, by Lemma \ref{i action}. A basis of the $I(1)$-invariants in $\tau_r$ consists of the vectors 
\[{\rm Id}\otimes\bigotimes_{j=0}^{f-1}x_j^{r_j}, \alpha\otimes\bigotimes_{j=0}^{f-1}y_j^{r_j},c_n^{p^l(r_l+1)}, \beta  c_n^{p^l(r_l+1)}, d_n^l, \beta d_n^l,\]
for $n \geq 1$, where 
\[d_n^l=\sum\limits_{\substack{\mu\in I_n}}\left(\begin{array}{cc}
\varpi^n & \mu\\
0        & 1
\end{array}\right)\otimes\bigotimes_{l\neq j=0}^{f-1}x_j^{r_j}\otimes x_l^{r_l-1}y_l,\]
by \cite[Theorem 1.2]{hen18}. By \cite[Lemma 3.6]{hen18},
\[\left(\begin{array}{cc}
a & b \\ \varpi c & d 
\end{array}\right) \cdot d_n^l = a^{r-2p^l}(ad)^{p^l} \cdot d_n^l,
\]
and once again it can be checked that there is no $I(1)$-invariant vector in $\tau_r$ with $I$-eigenvalue $a^{q-1-p^l}d^{r+p^l}$, where $t_n^{r+p^l}$ can be mapped under $\phi$, giving a contradiction. 

\section*{Acknowledgements} 

The second author would like to thank Council of Scientific and Industrial Research, Government of India (CSIR) and Industrial Research and Consultancy Centre, IIT Bombay (IRCC) for financial support.

\end{document}